\documentclass[a4paper, 11pt]{amsart} 
 \usepackage[foot]{amsaddr}

\usepackage{amsmath, amsthm, amssymb} 
\usepackage[utf8]{inputenc} 
\numberwithin{equation}{section}

\usepackage{lmodern} 
\usepackage[T1]{fontenc} 

\usepackage{enumerate} 
\usepackage{verbatim} 
\usepackage{color}
\usepackage{tabularx}

\usepackage[hidelinks]{hyperref} 

\newcommand{\Prob}{\mathbb{P}}
\newcommand{\Q}{\mathbb{Q}}
\newcommand{\R}{\mathbb{R}}
\newcommand{\E}{\mathbb{E}}
\newcommand{\I}{\mathbb{I}}
\newcommand{\N}{\mathbb{N}}

\newcommand{\calN}{\mathcal{N}}
\newcommand{\dd}{\mathrm{d}}
\renewcommand{\P}{\Prob}
\renewcommand{\S}{\mathbb{S}}
\renewcommand{\sp}[2]{#1 \cdot #2}
\newcommand{\ind}[1]{\I_{\{ #1\}}}

\theoremstyle{plain}
\newtheorem{theorem}{Theorem}[section]
\newtheorem{lemma}[theorem]{Lemma}

\newtheorem{corollary}[theorem]{Corollary}

\newtheorem{proposition}[theorem]{Proposition}
\newtheorem{conjecture}[theorem]{Conjecture}

\theoremstyle{remark}
\newtheorem{remark}[theorem]{Remark}

\title[Derivative martingale of the BBM in dimension $d \geq 1$]{Derivative martingale of the branching Brownian motion in dimension $d \geq 1$}

\author{Roman Stasiński}
\author{Julien Berestycki}
\address{Department of Statistics, University of Oxford, 24-29 St Giles’, Oxford OX1 3LB, UK}
\email{roman.stasinski@sjc.ox.ac.uk}
\email{julien.berestycki@stats.ox.ac.uk}

\author{Bastien Mallein}
\address{Université Sorbonne Paris Nord, LAGA, UMR 7539, F-93430, Villetaneuse, France}
\email{mallein@math.univ-paris13.fr}

\begin{document}
\begin{abstract}
We consider a branching Brownian motion  in $\R^d$. We prove that there exists a random subset $\Theta$ of $\S^{d-1}$ such that the limit of the derivative martingale exists simultaneously for all directions $\theta \in \Theta$ almost surely. This allows us to define a random measure on $\S^{d-1}$ whose density
is given by the derivative martingale. 

The proof is based on first moment arguments: we approximate the martingale of interest by a series of processes, which do not take into account particles that travelled too far away. We show that these new processes are uniformly integrable martingales whose limits can be made to converge to the limit of the original martingale.
\end{abstract}

\maketitle
\section{Introduction}

Consider  a branching Brownian motion in dimension $d \geq 1$. This is a particle system in which  independent particles move in $\R^d$ as Brownian motions and branch independently at rate 1 into two particles. This system behaves as a growing cloud of diffusing particles. Let us fix the notation. We denote by $\P_x$ the law of the branching Brownian motion starting from one particle at position $x\in\R^d$, (writing $\P$ for $\P_0$ for simplicity). For all times $t \geq 0$, we denote by $\calN_t$ the set of particles alive at time $t$, and for each particle $j \in \mathcal{N}_t$ and $s \leq t$, we write $X_s(j)$ for the position that $j$, or its ancestor at time $s$, occupied at time $s$. The natural filtration of the branching Brownian motion is denoted by $(\mathcal{G}_t, t \geq 0)$.

In \cite{Mallein2015}, Mallein studied the maximal displacement of this model, i.e. the quantity 
\[
R_t  = \max_{j \in \mathcal{N}_t} \| X_t(j) \|, \quad t\ge 0.
\]
He showed that as $t \to \infty$
\begin{equation}
  \label{eq:r_d_tightness}
R_t =\sqrt{2} t + \frac{d-4}{2\sqrt 2} \log t+O(1),
\end{equation}
 where $O(1)$ is a process $Y_t$ such that $\lim_{K\to \infty} \P( \sup_t |Y_t| >K)=0$, thus generalising a famous result of Bramson \cite{Bramson1978} for $d=1$.    
 \medskip

Imagine now that we want to know in which direction $D(t)$ is the particle at distance $R_t$ at time $t$. Under $\P_0$, the process is completely spherically symmetric and it is thus evident that the distribution of the direction $D(t)$ of this extremal particle is uniform on the sphere $\S^{d-1}$. However, if we first observe the process up to time $s$ and then try to guess the direction of the furthest particle at a later time $t$, the answer obviously depends on the configuration we observe at time $s$, even in the limit $t\to \infty$. Advantages gained or delays incurred early in a given direction are never forgotten. 

It is believed that almost surely, for all measurable sets $A \subset \S^{d-1}$
\[
  \lim_{s\to  \infty} \lim_{t\to \infty} \P(  D(t)\in A \ | \ \mathcal{G}_s ) =  \mu(A),
\]
where $\mu$ is a random probability measure which captures what happens early on in the life of the process. What should this measure be? 

To answer this question, it is instructive to look at the one-dimensional case. When $d=1$, it is well-known that the asymptotic behaviour of the extremal particles (i.e. particles within distance $O(t^{1/2})$ from the maximal displacement at time $t$) is mainly driven by the limit of the so-called derivative martingale, defined by
\begin{align*}
  Z_t^+ := \sum_{j \in \mathcal{N}_t} (\sqrt{2}t - X_t(j))e^{\sqrt{2}(X_t(j) - \sqrt{2}t)}.
\end{align*}
Although $(Z_t^+,t \geq 0)$ is known to be a non-uniformly integrable martingale, and clearly takes both positive and negative values, Lalley and Sellke \cite{LS1987} proved that it does have an almost sure limit $Z_\infty^+ := \lim_{t \to \infty} Z_t^+$ which is positive almost surely, and moreover
\begin{align*}
  \max_{j \in \mathcal{N}_t} X_t(j) - m_t - \frac{\sqrt{2}}{2} \log Z_\infty^+
\end{align*}
converges in law to a Gumbel random variable, where $m_t = \sqrt{2} t - \frac{3}{2\sqrt 2} \log t$. 

We introduce the maximal and minimal displacements, i.e. the largest displacement in the positive and negative direction:
\begin{equation*}
M^+_t := {}  \max_{j \in \mathcal{N}_t} X_t(j) \quad \text{ and } \quad
M^-_t := {}  \min_{j \in \mathcal{N}_t} X_t(j),
\end{equation*}
as well as the derivative martingale in the negative direction, which is the derivative martingale of the BBM $(-X_t(u), u \in \mathcal{N}_t)$. In other words, we set
\begin{align*}
  Z_t^- := \sum_{j \in \mathcal{N}_t} (\sqrt{2}t + X_t(j))e^{\sqrt{2}(X_t(j) + \sqrt{2}t)}
\end{align*}
and $Z_\infty^- := \lim_{t \to \infty} Z_t^-$. 
As far as we are aware,  the joint convergence in distribution of $(M^+_t, M^-_t)$ had not been considered until now.
\begin{theorem}
There exists a constant $c_\star$ such that for all $y,z \geq 0$ almost surely
\begin{multline*}
  \lim_{s \to \infty} \lim_{t \to \infty} \P\left(M^+_t - m_t \leq y, -M^-_t - m_t \leq z \ \bigg| \ \mathcal{G}_s \right)\\
  = \exp\left(-c_\star Z_\infty^+ e^{-\sqrt{2}y} -c_\star Z_\infty^- e^{-\sqrt{2}z}\right).
\end{multline*}
In other words, $\left(M^+_t-m_t - \frac{\sqrt{2}}{2} \log \left( c_\star Z_\infty^+\right),  -M^-_t - m_t - \frac{\sqrt{2}}{2} \log \left(c_\star Z_\infty^-\right)\right)$ converges in distribution towards a pair of independent Gumbel random variables with scale parameter $\frac{\sqrt{2}}{2}$.
\label{thm:main1d}
\end{theorem}

As a consequence, conditionally on $(Z_\infty^+,Z_\infty^-)$ the probability that the direction of the furthest particle at a large time is in the positive direction is proportional to $Z_\infty^+$.
\begin{corollary}
We have
\begin{align*}
\lim_{s \to \infty} \lim_{t \to \infty} \P\left(M^+_t > -M^-_t \ \Big| \ \mathcal{G}_s\right) = \frac{Z^+_\infty}{Z_\infty^{-} + Z^+_\infty} \quad \text{a.s.}
\end{align*}
\label{cor:1d}
\end{corollary}

\medskip

It is straightforward from the definition of the branching Brownian motion, that for all $\theta \in \S^{d-1}$, its projection on the direction $\theta$ (the process $(\sp{X_t(j)}{\theta}, u \in \mathcal{N}_t)$) is a branching Brownian motion in dimension one. Thus, for each $\theta\in \S^{d-1}$ we can define the derivative martingale of $X$ in direction $\theta$ as
\begin{align*}
  Z_t(\theta) := \sum_{j \in \mathcal{N}_t}(\sqrt{2}t - \sp{X_t(j)}{\theta})e^{\sqrt{2}(\sp{X_t(j)}{\theta} - \sqrt{2}t)}
\end{align*}
and  for each $\theta \in \S^{d-1}$, the limit $\lim_{t \to \infty} Z_t(\theta) = Z_\infty(\theta)$ exists a.s. 

Coming back to the direction $D(t)$ of the extremal particle, it is natural to think  that, as in dimension one, the random measure $\mu$ should give more mass to regions where $Z_\infty(\theta)$ is large. In fact, $\mu$ should have  a density given by the normalized version of $\theta \mapsto   Z_\infty(\theta)$. That is, for a measurable set $B \subset \S^{d-1}$, we would expect $\mu(B)= \left.\int_B Z_\infty(\theta) \sigma(\dd \theta)  \middle/  \int_{\S^{d-1}} Z_\infty(\theta) \sigma(\dd \theta)\right.$, where $\sigma(\dd \theta)$ stands for the surface measure of $\S^{d-1}$.

However, the problem is that we do not have a.s. existence of the limit $Z_\infty(\theta)$ for all $\theta \in \S^{d-1}$ simultaneously and so the above integrals are not {\it a priori} well defined. Observe for instance that by \eqref{eq:r_d_tightness} one has
\begin{align*}
  \inf_{\theta \in \S^{d-1}} Z_t(\theta) \leq -C( \log t)t^{(d-4)/2} \text{ with high probability,}
\end{align*}
hence the derivative martingale may be very small in exceptional directions, at least in dimension $d\geq 4$. This is due to the fact that in higher dimensions particles travel farther away from $0$ than in dimension one, which has the effect of lowering the value of $Z_t(\theta)$ in the (random) direction at which these far away particles are located. As a result, one cannot hope for uniform convergence to hold for the process $(Z_t(\theta))$. It is nonetheless the main object of this paper to show how one can make sense of the limit of the function $\theta \mapsto Z_t(\theta)$ in a weak sense. We also prove that almost surely the limit of $Z_t(\theta)$ actually exists for all $\theta$ in a set of full measure. Hence a rigorous meaning can be given to the associated measure $\mu$.

\medskip

In this article we prove the weak convergence of $(Z_t(\theta), \theta \in \S^{d-1})_{t \geq 0}$, seen as a random measure on the sphere. For two measurable functions $f,g: \ \S^{d-1} \mapsto \R$ we define 
\begin{align*}
\langle f,g \rangle := \int_{\S^{d-1}} f(\theta) g(\theta) \sigma(\dd \theta),
\end{align*}
where $\sigma$ is the Lebesgue measure on the sphere $\S^{d-1}$. We sometimes write $\langle f(\theta),g(\theta) \rangle$ to clarify how functions $f$ and $g$ depend on $\theta \in \S^{d-1}$.

The main result of this article is the following.
\begin{theorem}
Almost surely there exists a measurable subset $\Theta$ of $\S^{d-1}$ of full measure, such that $Z_\infty(\theta) := \lim_{t\to\infty} Z_t(\theta)$ exists for $\theta \in \Theta$, and for any bounded measurable function~$f$
\begin{align}
\lim_{t \to \infty} \langle Z_t, f \rangle = \langle Z_\infty, f \rangle \textrm{ a.s.},
\label{eq:thm_main}
\end{align}
writing $Z_\infty(\theta) = 0$ for all $\theta \not \in \Theta$. Additionally, $0 < \lim_{t \to \infty} \langle Z_t, 1 \rangle < \infty$ almost surely.
\label{thm:main}
\end{theorem}

Although we only consider the case of a {\it binary} branching mechanism (particles always split into two daughter particles), it would be straightforward to generalise our results to a situation in which an independent random number $L$ of children is produced at each branching event, at least under the assumption $\E(L (\log L)^{2+\delta}) < \infty$ for some $\delta > 0$. Note that it was shown by Yang and Ren \cite{YR2011} that, in the case of the one-dimensional branching Brownian motion, the limit of the derivative martingale is non-degenerate if and only if $\E(L (\log L)^2 ) < \infty$. This result was then extended by Chen~\cite{Chen2015} to the case of branching random walks and recently Boutaud and Maillard simplified and streamlined the proofs of these limit theorems in \cite{BM2019}. We believe Theorem~\ref{thm:main} would hold under similar optimal integrability conditions, but the proof would require additional control on the law of a Brownian motion conditioned to stay below a curve. 

Let us now formulate a conjecture regarding the full extremal point process, from which the predicted behaviour of $D(t)$ follows. This conjecture is a multidimensional version of the description of the extremal point process of the one-dimensional branching Brownian motion obtained by Arguin Bovier and Kistler \cite{ABK2013}, and Aïdékon, Berestycki, Brunet and Shi \cite{ABBS2013}. Recall from \cite[Theorem~1.1]{Mallein2015} that
\begin{align*}
r_t := \sqrt{2} t + \frac{d-4}{2\sqrt{2}} \log t
\end{align*} 
is, up to an $O(1)$ error, the median of the maximal displacement of the $d$-dimensional branching Brownian motion. We also define the direction of a particle $u$ at time $t$ by $D_t(u):=X_t(u)/\|X_t(u)\|$ for $t\ge 0, u\in \mathcal{N}_t$.
\begin{conjecture}
There exists $c_{d}^\star>0$ such that
\begin{align*}
	\lim_{t \to \infty} \sum_{u \in \mathcal{N}_t} \delta_{D_t(u), \|X_t(u)\| - r_t} = \mathcal{L}(\dd \theta, \dd x) \text{ in law},
\end{align*}
where $\mathcal{L}$ is a decorated Poisson point process that can be constructed as follows. Let $(\theta_j,\xi_j)_{j \geq 1}$ be the atoms of a Poisson point process with intensity $c_{d}^\star Z_\infty(\theta) \sigma(\dd\theta) e^{-\sqrt{2}x} \dd x$ and $(D_j, j \geq 1)$ be i.i.d. point processes on $\R$ with common distribution $\mathcal{D}$. Then
\begin{align*}
	\mathcal{L} = \sum_{j \geq 1} \sum_{x \in D_j} \delta_{\theta_j, \xi_j + x}.
\end{align*}
\label{con:extremal_process}
\end{conjecture}

To be more explicit, the decoration point measure $\mathcal{D}$ above can be constructed  as the weak limit of $\sum_{u \in \mathcal{N}_t} \delta_{\|X_t(u)\| - R_t}$ (the extremal process of moduli seen from the largest displacement) conditioned on $R_t \geq r_t + \frac{3}{2 \sqrt{2}} \log t$ (c.f. \cite{SZ2015} for a general result of convergence towards decorated Poisson point processes). In particular, $\mathcal{D}$ only charges $(-\infty, 0]$.

Let us discuss briefly some implications that would follows from Conjecture \ref{con:extremal_process}.
Firstly, an easy Poisson point process computation would yield that 
\begin{align*}
  \lim_{t \to \infty} \P(R_t -r_t \leq x) = \E\left[ \exp\left( - c_d^\star \langle Z_\infty, 1 \rangle e^{-\sqrt 2 x} \right) \right].
\end{align*}
This is the multidimensional version of \cite{LS1987} that gives the convergence in law of the maximum of the branching Brownian motion.
Similarly, it would imply the following convergence for the law of the direction of the furthest particle at time $t$:
\[
  \lim_{s \to \infty} \lim_{t \to \infty}  \P\left( D(t) \in B \ \middle| \ \mathcal{G}_s\right)  = \frac{1}{\langle Z_\infty,1\rangle} \int_B Z_\infty(\theta) \dd \theta \quad \text{ a.s. }, B\subseteq \mathbb{S}^{d-1}.
\]

\section{Proof strategy}
\label{sec:pfStrategy}

Let us now review briefly how these results are usually proved in dimension $d=1$. The idea is to get rid of all particles that ever reach level $\sqrt{2} t + A$ at some time $t$ (this is sometimes referred to as a {\it shaving argument}). However, as we push the barrier away by letting $A\to \infty$  the probability that any particle ever hits the barrier decreases to zero. More formally, one introduces the martingale
\begin{align*}
  Z^{A}_t := \sum_{j \in \mathcal{N}^{A}_t} (\sqrt{2} t+A - X_t(j)) e^{\sqrt{2} (X_t(j) - \sqrt{2} t)},
\end{align*}
where $\mathcal{N}^{A}_t = \{ j \in \mathcal{N}_t : X_s(j) \leq \sqrt{2} s + A , s \leq t\}$. This martingale is non-negative (and uniformly integrable) and therefore converges to some $Z_\infty^A$. As in dimension one 
\begin{align*}
  \sup_{t \geq 0} \sup_{j \in \mathcal{N}_t} |X_t(j)| - \sqrt{2} t < \infty \qquad \text{ almost surely,}
\end{align*}
hence taking $A$ large enough ensures that no particle is killed with high probability. This proves that the derivative martingale converges and that almost surely $Z_\infty = \lim_{A \to \infty} Z_\infty^A$. In larger dimensions, however, one has
\begin{align*}
  \sup_{t \geq 0} \sup_{j \in \mathcal{N}_t} \|X_t(j)\| - \sqrt{2} t = \infty,
\end{align*}
and this is the moment where the standard argument breaks.

To overcome this difficulty we need to introduce a different way of removing particles that fly too high. This is done by killing particles that reach a curved boundary $\sqrt{2} t + (\phi(t) \vee A)$ at some time $t$, with $\phi$ a well-chosen non-decreasing function. In particular, if $\phi$  grows fast enough, we can ensure  that no particle will be removed with high probability by letting $A\to \infty$. The difficulty is then to find an analogue of the martingale $Z^A$ for this curved boundary.

The outline of the paper is as follows:  In Section~\ref{sec:bmKilled} we  study the standard Brownian motion killed when hitting the barrier $t \mapsto \phi(t) \vee A$. We prove in particular existence of some function $R^\phi$ allowing us to describe the Brownian motion conditioned to stay below $\phi \vee A$ as a Doob $h$-transform. Then in Section~\ref{sec:BBM_and_shaved_martingale} we prove that with high probability all particles of the multidimensional BBM do not escape a ball of an increasing radius, construct a family of martingales that we use to approximate $(Z_t(\theta))_{\theta \in \S^{d-1}}$, study their asymptotic behaviour, and prove Theorem~\ref{thm:main}. Finally, in Section~\ref{sec:1dim} we treat the one-dimensional case and look at the joint law of the leftmost and rightmost particles in the branching Brownian motion.

\section{Brownian motion conditioned to avoid \texorpdfstring{$\phi(t)$}{phi(t)}}
\label{sec:bmKilled}

To prove Theorem \ref{thm:main}, as explained in Section~\ref{sec:pfStrategy}, we will need some estimates on the one-dimensional Brownian motion conditioned to stay below a curve. In this section we gather several results on this process, using Doob's $h$-transform theory.

Let $\phi$ be a continuous function $[0,\infty) \to \R$ such that $\phi(t) = o(t^{1/2-\epsilon})$ for some $\epsilon>0$. We start by studying the Brownian motion conditioned not to hit the function $\phi$ until some finite time $t$. As the fluctuations of $B_t$, which are of order $t^{1/2}$, are much larger than $\phi(t)$, we expect that for $1 \ll s \ll t$ the process $B$ on $[0,s]$ conditioned on not hitting $\phi$ until time $t$ behaves roughly like a Bessel process (a Brownian motion conditioned not to hit 0).

More precisely, we introduce the relevant non-negative $h$-transform function $R^\phi$ in Lemma \ref{lm:R_def}. Defined as the renormalized probability of avoiding $\phi$, it makes
\begin{align*}
(R^\phi(B_t, t) \ind{\forall s <  t, B_s \leq \phi(s)})_{t \geq 0}
\end{align*}
a martingale. In other words, $R^\phi$ is a harmonic function for the Markov process $(B_t,t)$ confined to $\{(x,t) : x \leq \phi(t)\}$. The Doob $h$-transform obtained then describes a Brownian motion conditioned to stay below $\phi$; we are going to denote the corresponding measure as $\P^\phi$. It will  also be important to show that there exists $C>0$ such that $R^\phi(x,t) \approx -Cx$ as $x \to -\infty$, as this will entail the `Bessel-like' behaviour we want.

The function $R^\phi$ will then be used to define approximations of the derivative martingale of the one-dimensional branching Brownian motion. Indeed, we wish to define a uniformly integrable martingale that approximates the derivative martingale
\begin{align}
Z_t = \sum_{j \in \mathcal{N}_t} (\sqrt{2} t - X_t(j)) e^{\sqrt{2} (X_t(j) - \sqrt{2} t)},
\label{eq:derivative_martingale}
\end{align}
that would be of the form
\begin{align}
\sum_{j \in \mathcal{N}^\phi_t} H(X_t(j) - \sqrt{2}t,t)e^{\sqrt{2} (X_t(j) - \sqrt{2} t)}
\label{eq:prototype_shaved}
\end{align}
where $H$ is some function and $\mathcal{N}^{\phi}_t = \{ j \in \mathcal{N}_t : X_s(j) \leq \sqrt{2} s + \phi(s) , \ \forall s \leq t\}$ (so that the sum in \eqref{eq:prototype_shaved} is taken only over the particles that did not hit the boundary $\sqrt{2}s + \phi(s)$). 
Assuming that \eqref{eq:prototype_shaved} is a martingale is equivalent to assuming that $(H(B_t, t) \ind{\forall s <  t, B_s \leq \phi(s)})_{t \geq 0}$ is itself a martingale. Hence, setting $H(x,t) = C^{-1}R^\phi(x,t)$ gives the desired approximation of \eqref{eq:derivative_martingale}.

\medskip 

The rest of the section is organised as follows. In Lemma \ref{prop:finite_time_restriction} we characterise the measure $\P^\phi$ as a limit of conditional distributions. In Lemma \ref{lm:1dim_change} we define a new measure $\P^V$, that corresponds to a Girsanov transform adding a drift $\sqrt{2}$ applied to a process with law $\P^\phi$. That is, we can interpret $\P^V$ as a measure of a Brownian motion with a drift $\sqrt{2}t$ conditioned on never hitting $\sqrt{2}t + \phi(t)$. In Lemma \ref{lm:Bessel_type_fluctuations}  we formalize the `Bessel-like' behaviour under $\Prob^\phi$. Finally, in Lemma \ref{lm:R_asymptotics} we study asymptotics of $R^\phi(x,t)$.

\subsection{Brownian motion and non-linear barriers}
For any continuous function $\phi: [0, \infty) \to \R$ set $\tau_\phi = \inf \{u >	 0 \ : B_u \geq \phi(u) \}$. The aim of this section is to give a precise asymptotic of the quantity $\P_x( \tau_\phi > t)$ as $t \to \infty$ for $\phi$ in a certain class. We are also interested in the dependence of this probability on the shift of $\phi$, i.e. we are going to consider functions $\phi_t(u) := \phi(t+u)$. 

It is well-known that if $\phi$ grows slower than $t^{1/2}$ as $t \to \infty$, in a sense to be made precise soon, then $\tau_\phi < \infty$ a.s. and $\Prob(\tau_\phi > t)$ decays as $t^{-1/2}$. More precisely, Uchiyama proved the following upper bound.

\begin{theorem}[\cite{Uchiyama1980}, Proposition 3.1. (i)]
	Let $\phi$ be a $\mathcal{C}^1$-class increasing function such that $\phi(0) = 0$ and $\lim_{t \to \infty} \phi(t) t^{-1/2}=0$. If
	\begin{align*}
	\phi(u) - \frac{u}{t} \phi(t) \geq 0 \quad \text{ for } 0<u<t,
	\end{align*}
	then there exists a constant $C$ such that for all $x \in \R$ and $t>1$
	\begin{align}
	\P_x(\tau_\phi > t) \leq \frac{1+|x|}{t^{1/2}} \exp\left( \frac{\sqrt{2 \pi}}{4} \int_1^t \frac{\phi(u)}{u^{3/2}} \dd u + C \int_1^t \frac{\phi(u)^2}{u^2} \dd u \right).
	\label{eq:Uchiyama}
	\end{align}
	\label{thm:Uchiyama}
\end{theorem}

Novikov \cite{Novikov1997} obtained a precise asymptotic of $\P(\tau_\phi > t)$ as $t \to \infty$, expressed as a function of the law of $B_{\tau_\phi}$.

\begin{theorem}[\cite{Novikov1997}, Theorem 2]
	If $\phi$ is a continuous non-decreasing function such that $\int_1^\infty \phi(t)t^{-3/2}dt < \infty$ and  $\phi(0) > 0$, then
	\begin{align*}
	\lim_{t\to \infty} \sqrt{t} \Prob(\tau_\phi >t) = \sqrt{\frac{2}{\pi}} \E B_{\tau_\phi} < \infty.
	\end{align*}
	\label{thm:Novikov_existence}
\end{theorem}

We apply these two theorems to define and give the first property of the aforementioned function $R^\phi$, which will be a key object of interest in the rest of the article. We will restrict ourselves to functions $\phi$ satisfying the following assumptions:
\begin{multline}
  \label{eqn:hypPhi}
  \tag{H}
  \phi \text{ increasing, concave, } \mathcal{C}^1\text{-class with $\phi(0) > 0$,}\\
  \text{ and there exists $\alpha \in (0,1/2)$ such that } \lim_{t \to \infty} \frac{\phi(t)}{t^\alpha} = 0,
\end{multline}
that we refer to as assumption \eqref{eqn:hypPhi}.

\begin{lemma}
Let $\phi$ be a function satisfying \eqref{eqn:hypPhi}. Then the following limit exists for all $t \geq 0$ and $x \in \R$:
	\begin{align}
	R^\phi(x,t) := \sqrt{\frac{\pi}{2}}\lim_{s \to \infty} \sqrt{s} \P_{x}(\tau_{\phi_t} > s).
	\label{eq:R_def}
	\end{align}
	Moreover, there exists $C > 0$ such that for all $t \geq 0$ and $x \leq \phi(t)$,
	\begin{equation*}
	R^\phi(x,t) \leq C (1 + (\phi(t)-x)).
	\end{equation*}
	Finally, $\left(R^\phi(B_t, t) \ind{\tau_\phi > t}\right)_{t \geq 0}$ is a martingale.
	\label{lm:R_def}
\end{lemma}

The idea of using the renormalized survival probability to define an $h$-transform is classical. Here we draw inspiration from \cite{BD1994}  (in which the law of a random walk conditioned to stay positive was constructed). There, as in the present work, we condition a random process not to hit some region (in our case the process of interest is $(B_t,t)_{t \geq 0}$ and the region to avoid is defined by $J = \{(x,t) : x \geq \phi(t)\}$).

In this setting the probability that the process $(B_t,t)_{t \geq 0}$ never hits $J$ is equal to $0$, irrespectively of its starting position $(x_0,t_0)$. As a result, to define the $h$-transform allowing the definition, in the sense of Doob, of the Brownian motion conditioned never to hit $J$, it is reasonable to renormalise the probability not to hit the region $J$ for $t$ units of time by $t^{1/2}$ so that the limit, that we denote by $R^\phi(x_0,t_0)$, is non-degenerate. It remains to check that the function $R^\phi$ which we defined is indeed a harmonic function for $(B_t,t)$ on the domain $J^c$, i.e. that $\left(R^\phi(B_t, t) \ind{\tau_\phi > t}\right)_{t \geq 0}$ is a martingale.

\begin{proof}
	The assumptions on $\phi$ guarantee that Theorem \ref{thm:Uchiyama} and Theorem \ref{thm:Novikov_existence} can be applied to the function $\phi_t$ for all $t \geq 0$. We note that for all $t \geq 0$, $s \geq 0$ and $x \geq \phi(t)$, we have
	\begin{align*}
	\P_{x}(\tau_{\phi_t} > s) = 0.
	\end{align*}	
	Applying Theorem \ref{thm:Novikov_existence} to the function $\phi_t$, we deduce that for all $x,t$ such that $x < \phi(t)$,
	\begin{align*}
	R^\phi(x,t) = \E\left( B_{\tau_{\phi_t-x}} \right) \ \in (0, \infty),
	\end{align*}
	which proves that $R^\phi$ is well-defined and finite. Additionally, using that $\phi$ is concave, and hence that $\phi(t+u) - \phi(t) \leq \phi(u) - \phi(0)$, we observe that for all $x \in \R$, $t \geq 0$ and $s > 0$, we have
	\begin{align*}
	&\sqrt{s} \P_{x}(\tau_{\phi_t} > s)
	= \sqrt{s} \P_{x-\phi(t)} \left(\tau_{\phi_t-\phi(t)} > s \right)
	\leq \sqrt{s} \P_{x - \phi(t)} \left( \tau_{\phi - \phi(0)} > s \right).
	\end{align*}
	Using Theorem \ref{thm:Uchiyama}, and observing that the exponential term in bound \eqref{eq:Uchiyama} is increasing in $t$, and hence may be bounded from above by its limit  as $t \to \infty$, we obtain for $x \leq \phi(t)$ and $s \geq 0$,
	\begin{align}
	\sqrt{s} \P_{x - \phi(t)} \left( \tau_{\phi - \phi(0)} > s \right) \leq {} C (1 + (\phi(t)-x)),
	\label{eqn:boundForR}
	\end{align}
	where $C> 0$ is a constant that does not depend on $x,t,s$.
	
	Thanks to this bound, we can now prove that $(R^\phi(B_t, t) \ind{\tau_\phi > t}, t \geq 0)$ is a martingale using the dominated convergence theorem. Indeed, using the Markov property, note that it is enough to prove that for all $t,s \geq 0$ and $x \leq \phi(t)$,
	\[
	  R^\phi(x,t) = \E_{x} \left[ R^\phi(B_s,t+s) \ind{\tau_{\phi_t} > s} \right].
	\]
	Observe that by the Markov property of the Brownian motion, for all $r \geq 0$
	\begin{align}
	\P_{x}(\tau_{\phi_t} > s + r)
	= \E_{x}\left[\ind{\tau_{\phi_t} > s} \P_{B_s}(\tau_{\phi_{t+s}} > r)\right].
	\label{eq:R_martingale}
	\end{align} 
	By definition, we have $\sqrt{r \pi/2}\P_{x}(\tau_{\phi_t} > s + r) \to R^\phi(x,t)$ as $r \to \infty$, and similarly, we have
	\begin{align*}
	\lim_{r \to \infty} \sqrt{r\pi/2} \P_{B_s}(\tau_{\phi_{t+s}} > r) = R^\phi(B_s,t+s) \text{ a.s.}
	\end{align*}
	We now observe that by \eqref{eqn:boundForR} we can bound $\sqrt{r}\P_{B_s}(\tau_{\phi_{t+s}} > r)$ uniformly in $r \geq 0$ by $C(1 + |B_s| + |\phi(t+s)|)$. This quantity being integrable, letting $r \to \infty$, and applying Lebesgue's dominated convergence theorem in \eqref{eq:R_martingale} we get
	\begin{align*}
	R^\phi(x,t) &=
	\E_{x}\left[ R^\phi(B_s,t+s) \ind{\tau_{\phi_t} > s}\right],
	\end{align*}
	which completes the proof.
\end{proof}

As mentioned above, the function $R^\phi$ can be used to construct the Brownian motion conditioned to stay below $\phi$ in the sense of Doob, as a process with law $\P^\phi$ defined by
\begin{align*}
\left. \frac{\dd\P^\phi}{\dd\P}\right|_{\mathcal{F}_t} := \frac{R^\phi(B_t, t)}{R^\phi(0,0)} \ind{\tau_\phi > t},
\end{align*}
using the fact that $R^\phi(B_t, t) \ind {\tau_\phi > t}$ is a non-negative martingale with mean $R^\phi(0,0)$. Law $\P^\phi$ corresponds to the limit of the law of the Brownian motion on the time interval $[0,t]$ conditioned on $\tau_\phi > s$ when $s \to \infty$. More precisely, it can be characterized in the following way.
\begin{proposition}
	Assume that $\phi$ satisfies \eqref{eqn:hypPhi}. For any $t>0$ and $A \in \mathcal{F}_t$, 
	\begin{align*}
	\P^\phi (A)  = \lim_{s \to \infty} \P(A \ | \ \tau_\phi>s).
	\end{align*}
	\label{prop:finite_time_restriction}
\end{proposition}
The proof of Proposition \ref{prop:finite_time_restriction} is inspired by ideas from the proof of Theorem~1 in \cite{BD1994}.
\begin{proof}
	Let $A \in \mathcal{F}_t$. We observe that
	\[
	  \P(A \ | \ \tau_\phi > s) = \frac{\P(A,\tau_\phi > s)}{\P(\tau_\phi > s)} = \frac{\E(\I_A \ind{\tau_\phi > t} \P_{B_t}(\tau_{\phi_t}>s-t))}{\P(\tau_\phi>s)}.
	\]
	Then by \eqref{eq:R_def}, we have that $\lim_{s \to \infty} \sqrt{s \frac{\pi}{2}} \P(\tau_\phi>s) = R^\phi(0,0)$ and
	\[
	  \lim_{s \to \infty} \sqrt{s \frac{\pi}{2}}\P_{B_t}(\tau_{\phi_t}>s-t)) = R^\phi(B_t,t) \quad \text{a.s.}
	\]
	Moreover, using \eqref{eqn:boundForR}, we can apply Lebesgue's dominated convergence theorem to obtain
	\[
	  \lim_{s \to \infty} \sqrt{s \frac{\pi}{2}}\E(\I_A \ind{\tau_\phi > t} \P_{B_t}(\tau_{\phi_t}>s-t)) = \E(\I_A \ind{\tau_\phi>t} R^\phi(B_t,t)).
	\]
	As a result, we have
	\[
	  \lim_{s \to \infty}  \P(A \ | \ \tau_\phi > s) = \frac{1}{R^\phi(0,0)}\E(\I_A \ind{\tau_\phi>t} R^\phi(B_t,t)) = \P^\phi(A),
	\]
	by definition.
\end{proof}

To complete the section, note that one can make a Girsanov-type change of measure to give the Brownian motion we consider a linear drift. This additional change of measure will be used when working with a multidimensional BBM. In particular, in Lemma \ref{lm:spinaldecomposition} we describe a decomposition of the size-biased law of the BBM with a spine particle that behaves similarly to a Brownian motion with drift $\sqrt{2}$ conditioned not to hit $\sqrt{2} t + \phi(t)$ for all $t \geq 0$.

More precisely, we introduce the hitting time
\begin{align*}
\tilde \tau_\phi :=  \inf \{u > 0 \ : B_u \geq \sqrt{2}u + \phi(u) \}
\end{align*}
and the process
\begin{align*}
V_t := \frac{R^\phi(B_t-\sqrt{2}t,t)}{R^\phi(0,0)} \ind{\tilde \tau_\phi > t} e^{\sqrt{2}B_t - t}.
\end{align*}
The following result then holds.

\begin{lemma}
Assuming that $\phi$ satisfy \eqref{eqn:hypPhi}, $(V_t, t \geq 0)$ is a mean one martingale. Defining $\P^V$ by $\frac{\rm{d} \P^V}{\rm{d} \P}\big|_{\mathcal{F}_t} := V_t$, $\P^V$ is a probability measure corresponding to the law of a Brownian motion with drift $\sqrt{2}$ conditioned to stay below $\sqrt{2}t + \phi(t)$ at all times $t \geq 0$ (in the sense of Proposition \ref{prop:finite_time_restriction}).
\label{lm:1dim_change}
\end{lemma}

\begin{proof}
Set $Y_t := e^{\sqrt{2}B_t - t}$. It is then well-known that $Y$ is a $\P$-martingale and that the law $\tilde \P = Y \cdot \P$ corresponds to the law of a Brownian motion with drift $\sqrt{2}$, by  Girsanov's theorem. Observe that
\begin{align*}
\left.\frac{\dd \P^V}{\dd \P}\right|_{\mathcal{F}_t} = \frac{R^\phi(B_t - \sqrt{2}t ,t) \ind{\tau_\phi > t}}{R^\phi(0,0)}e^{\sqrt{2}B_t - t}
= \frac{R^\phi(B_t - \sqrt{2}t ,t) \ind{\tilde \tau_\phi > t}}{R^\phi(0,0)} \left. \frac{\dd \tilde{\P}}{\dd \P}\right|_{\mathcal{F}_t}.
\end{align*}

Using that under $\tilde{\P}$, $(B_t-\sqrt{2} t, t \geq 0)$ is a Brownian motion, we obtain immediately from Lemma \ref{lm:R_def} that $(R^\phi(B_t - \sqrt{2}t,t) \ind{\tilde \tau_\phi > t}, t \geq 0)$ is a non-negative $\tilde{\P}$-martingale, and therefore that $V$ is a $\P$-martingale.

Additionally, we have that
\[
  \left.\frac{\dd \P^V}{\dd \tilde{\P}}\right|_{\mathcal{F}_t} = \frac{R^\phi(B_t - \sqrt{2}t,t) \ind{\tau_\phi > t}}{R^\phi(0,0)},
\]
hence by Proposition \ref{prop:finite_time_restriction} we have that under $\P^V$, $(B_t - \sqrt{2} t, t \geq 0)$ is a Brownian motion conditioned on not hitting the curve $\phi$, which completes the proof.
\end{proof}

\subsection{Behaviour of the conditioned process}

We describe here the behaviour of the process $B$ under the law $\P^\phi$. We prove that for the Brownian motion conditioned to stay below $\phi$, the process localizes at time $t$ at position $-t^{1/2 + o(1)}$. In other words, for any $\epsilon \in (0, 1/2)$, for all $t$ large enough one has $t^{1/2-\epsilon} < -B_t < t^{1/2 + \epsilon}$ $\P^\phi$-a.s. This result is similar to what happens with the Bessel process, i.e. as the Brownian motion typically has $\sqrt{t}$ fluctuation, conditioning it to stay below $0$ or $o(t^{1/2-\epsilon})$ does not make a difference, asymptotically.                                                                            \begin{lemma}
  Let $\phi$ be a function satisfying \eqref{eqn:hypPhi}. We have
  \[
    \lim_{t \to \infty} \frac{\log (-B_t)}{\log t} = \frac{1}{2} \quad \P^\phi-\text{a.s.},
  \]
  i.e. $B_t = -t^{1/2+o(1)}$ as $t \to \infty$, $\P^\phi$-a.s.
	\label{lm:Bessel_type_fluctuations}
\end{lemma}

We split this lemma into several pieces. We begin with an upper bound for the probability for $B$ to be close to $\phi(t)$ at time $t$ under the law $\P^\phi$.
\begin{lemma}
\label{lem:smallValue}
Let $\phi$ be a function satisfying \eqref{eqn:hypPhi}. There exists $C > 0$ such that for all $t, x \geq 0$ we have
\[
  \P^\phi( B_t \geq \phi(t) - x) \leq C \left( \frac{1+x}{(1+t)^{1/2}}\right)^{3}.
\]
\end{lemma}

\begin{proof}
Let $x \geq 0$ and $t \geq 1$. Using the definition of $\P^\phi$ we have
\begin{align*}
  \P^\phi(B_t \geq \phi(t) - x) = {} & \E\left( R(B_t,t) \ind{B_t \geq \phi(t) - x, \tau_\phi > t} \right)\\
  \leq {} & \sup_{z \in [0,x]} R(\phi(t)-z, t) \P\left(B_t \geq \phi(t) - x, \tau_\phi > t \right)\\
  \leq {} & C (1 + x) \P(B_t \geq \phi(t) - x, \tau_\phi>t),
\end{align*}
by \eqref{eqn:boundForR}. By the Markov property at time $t/2$, we have
\begin{align*}
  &\P(B_t \geq \phi(t) - x, \tau_\phi>t) \\
  \leq {} &\P(\tau_\phi > t/2) \sup_{z \in \R} \P_z(B_{t/2} \geq \phi(t) - x, B_s \leq \phi(t/2+s), s \leq t/2)\\
  \leq {} &C t^{-1/2} \sup_{z \in \R} \P_z( B_{t/2} \geq \phi(t) - x, B_s \leq \phi(t), s \leq t/2),
\end{align*}
using Theorem \ref{thm:Uchiyama}.

We now use time-reversal of the Brownian motion, observing that under $\P_z$, $\hat{B}_s := B_{t/2} - B_{t/2-s}$ is a Brownian motion started from $0$. We use it to estimate
\begin{align*}
  & \sup_{z \in \R} \P_z( B_{t/2} \geq \phi(t) - x, B_s \leq \phi(t), s \leq t/2)\\
  = {} & \sup_{z \in \R} \P_z(\hat{B}_{t/2}+ z \geq \phi(t)-x, \hat{B}_{t/2}+ z - \hat{B}_{s} \leq \phi(t), s \leq t/2)\\
  \leq {} & \sup_{z \in \R} \P_z(\hat{B}_{t/2} \geq \phi(t)- z - x,\hat{B}_{t/2}+z \leq \phi(t), \hat{B}_s \geq -x, s \leq t/2)\\
  = {} & \sup_{z' \in \R} \P(\hat{B}_{t/2} \in [z',z'+x], \hat{B}_s \geq -x, s \leq t/2)\\
  \leq {} & \P(\hat{B}_s \geq -x, s \leq t/4) \sup_{z \in \R} \P(\hat{B}_{t/4} \in [z,z+x]),
\end{align*}
using the Markov property at time $t/4$. Then, using again Theorem \ref{thm:Uchiyama}, there exists $C > 0$ such that for all $x \geq 0$ and $t \geq 1$,
\[
  \P(\hat{B}_s \geq -x, s \leq t/4) \leq C (1 + x)/t^{1/2}.
\]
Additionally, we have $ \P(\hat{B}_{t/4} \in [z,z+x]) \leq \sqrt{\frac{2}{\pi t}}x$ for all $z \in \R$, noting that the density of $\hat{B}_{t/4}$ is bounded by $\sqrt{\frac{2}{\pi t}}$. Finally, we obtain the existence of $C>0$ such that for all $t,x \geq 0$
\[
  \P^\phi(B_t \geq \phi(t) - x) \leq C \frac{(1+x)^3}{(1+t)^{3/2}}. \qedhere
\]

\end{proof}

We now use this result to bound from below the asymptotic behaviour of $\frac{\log (-B_t)}{\log t}$.
\begin{lemma}
\label{lm:BeTypeFluc:lb}
Given $\phi$ a function satisfying \eqref{eqn:hypPhi}, we have
\[
  \liminf_{t \to \infty} \frac{\log(-B_t) }{\log t} \geq \frac{1}{2} \quad \text{a.s.}
\]
\end{lemma}

\begin{proof}
To prove this result, we begin by using the Borel-Cantelli lemma to show that almost surely, for all $\gamma < 1/2$,
\begin{equation}
  \label{eqn:aim:discreter}
  \liminf_{n \to \infty} \frac{\log(-B_{t_n})}{\log t_n} \geq \gamma \quad \text{a.s.}
\end{equation}
along a well-chosen sequence $t_n$ growing to $\infty$. We then use the observation that with high probability the Brownian motion between times $t_{n}$ and $t_{n+1}$ stays within a distance $O(t_{n+1}-t_n)^{1/2}$ from $B_{t_n}$. Therefore, as long as $(t_{n+1}-t_n)^{1/2}/t_n^\gamma \to 0$, we can extend \eqref{eqn:aim:discreter} to any sequence growing to $\infty$, which completes the proof.

Let $\gamma < 1/2$. We assume without loss of generality that $\gamma$ is close enough to $1/2$, such that $\phi(t) = o(t^\gamma)$. Using Lemma \ref{lem:smallValue} we have
\[
  \P^\phi(B_t \geq - t^\gamma) \leq C t^{3\left(\gamma - \frac{1}{2}\right)}.
\]
As a result, setting $t_n = n^{\frac{5}{6(1-2\gamma)}}$, we have
\begin{equation}
  \label{eqn:onebound}
  \P^\phi\left( \frac{\log (-B_{t_n})}{\log t_n} \leq \gamma \right) \leq C n^{- 5/4},
\end{equation}
hence, by the Borel-Cantelli lemma,
\[
  \liminf_{n \to \infty}  \frac{\log (-B_{t_n})}{\log t_n} \geq \gamma \quad \text{a.s.}
\]

To complete the proof we now need to bound the maximal displacement of the Brownian motion in the time intervals $[t_{n},t_{n+1}]$. Write $A  = \frac{5}{6(1-2\gamma)}$ so that $t_n = n^A$ and compute for $n \in \N$
\begin{align*}
  &\P^\phi\left( \sup_{s \in [t_n,t_{n+1}]} B_s \geq -t_n^{\gamma}/2, B_{t_n} \leq -t_n^\gamma \right)\\
  = {} & \E\left(R^\phi(B_{t_{n+1}},t_{n+1}) \ind{\tau_\phi > t_{n+1},B_{t_n} \leq -t_n^\gamma} \ind{\sup_{s \in [t_n,t_{n+1}]} B_s \geq -t_n^{\gamma}/2} \right).
\end{align*}

We can decompose this quantity depending on whether $B_{t_{n+1}}$ is smaller or larger than $-t_{n+1}^{2/3}$. Observe that for all $t \geq 1$ we have
\begin{align*}
  \E\left(R^\phi(B_{t},t) \ind{B_t < -t^{2/3}}\right)
  \leq {} &C \E\left( \left( 1 + |B_{t}| + |\phi(t)|\right) \ind{B_{t} < -t^{2/3}} \right)\\
  \leq {} &C \E\left( |B_{t}| \ind{B_{t} < - t^{2/3}} \right) \leq C e^{- \frac{t^{4/3}}{2t}},
\end{align*}
using that $|\phi(t)| = o(t^{2/3})$ as $t \to \infty$ and integrating with respect to the Brownian density. Thus, there exists $C > 0$ such that for all $n \in \N$ 
\begin{equation*}
  \E\left(R^\phi(B_{t_{n+1}},t_{n+1}) \ind{B_{t_{n+1}} < -t_{n+1}^{2/3}}\right)
  \leq C\exp\left(-t_{n+1}^{1/3}/2\right).
\end{equation*}
Hence, using that there exists $C > 0$ such that $R^\phi(x,t_{n+1}) \leq C t_{n+1}^{2/3}$ for all $x \geq - t_{n+1}^{2/3}$, 
\begin{multline}
  \E\left(R^\phi(B_{t_{n+1}},t_{n+1}) \ind{\tau_\phi > t_{n+1},B_{t_n} \leq -t_n^\gamma} \ind{\sup_{s \in \left[t_n,t_{n+1}\right]} B_s \geq -t_n^{\gamma}/2} \right)\\
  \leq   C t_{n+1}^{2/3} \P\left(\tau_\phi > t_{n+1},\sup_{s \in [t_n,t_{n+1}]} B_s \geq -t_n^{\gamma}/2 , B_{t_n} \leq -t_n^\gamma \right) \\
  + C\exp\left(-t_{n+1}^{1/3}/2\right). 
  \label{eqn:intermediate}
\end{multline}

We now bound $\P\left(\tau_\phi > t_{n+1},\sup_{s \in [t_n,t_{n+1}]} B_s \geq -t_n^{\gamma}/2 , B_{t_n} \leq -t_n^\gamma \right)$. Using the Markov property at time $t_n$ we have
\begin{multline*}
  \P\left(\tau_\phi > t_{n+1},\sup_{s \in [t_n,t_{n+1}]} B_s \geq -t_n^{\gamma}/2 , B_{t_n} \leq -t_n^\gamma \right)\\
  \leq \E\left( G_n(B_{t_n}) \ind{\tau_\phi > t_n}   \ind{B_{t_n} \leq -t_n^\gamma}\right),
\end{multline*}
where $G_n(x) := \P_x\left(\ind{\sup_{s \leq t_{n+1}-t_n} B_s \geq -t_n^{\gamma}/2} \right)$. As $G_n(x)$ is non-decreasing in $x$, using the Brownian scaling, for all $x \leq -t_n^\gamma$ we have
\[
  G_n(x) \leq G_n(-t_n^\gamma) = \P_{-1} \left( \sup_{s \leq (t_{n+1}-t_n)/t_n^{2\gamma}} B_s \geq -1/2  \right).
\]
By definition of $A$ and $t_n$ we note that
\[
  \frac{t_{n+1} - t_n}{t_n^{2\gamma}} = \frac{(n+1)^A - n^A}{n^{2 A \gamma}} \sim A n^{A-1-2 \gamma A} = A n^{-1/6} \text{ as } n \to \infty.
\]
As the maximum of a Brownian motion on $[0,s]$ is distributed as the absolute value of a Gaussian random variable with parameter $s$, and using standard Gaussian estimates, we have
\begin{align*}
  \P_{-1} \left( \sup_{s \leq (t_{n+1}-t_n)/t_n^{2\gamma}} B_s \geq -1/2  \right) &\leq \P\left( \sup_{s \leq C A n^{-1/6}} B_s \geq 1/2\right) \\
  &\leq \frac{1}{\sqrt{\pi C A n^{-1/6}}} \exp\left( - \frac{1}{8C A n^{-1/6}} \right).
\end{align*}
Thus we deduce that for all $x \leq -t_n^\gamma$ we have $G_n(x) \leq C e^{-cn^{1/6}}$. Since $t_n$ has polynomial growth, we therefore obtain from \eqref{eqn:intermediate} that there exists $C,\delta > 0$ such that
\begin{align*}
  \E\left(R^\phi(B_{t_{n+1}},t_{n+1}) \ind{\tau_\phi > t_{n+1},B_{t_n} \leq -t_n^\gamma} \ind{\sup_{s \in \left[t_n,t_{n+1}\right]} B_s \geq -t_n^{\gamma}/2} \right) \leq C e^{-n^\delta}.
\end{align*}

We now conclude, using \eqref{eqn:onebound}, that
\begin{align*}
  &\sum_{n \in \N} \P^\phi(\sup_{s \in [t_n,t_{n+1}]} B_s \geq - t_n^{\gamma}/2)\\
   \leq {} &\sum_{n \in \N} \P^\phi(B_{t_n} \geq -t_n^\gamma) + \sum_{n \in \N} \P^\phi\left(B_{t_n} \leq -t_n^\gamma, \sup_{s \in [t_n,t_{n+1}]} B_s \geq - t_n^\gamma/2\right)\\
  \leq & C \sum_{n \in \N} n^{-5/4} + \sum_{n \in \N} e^{-n^{\delta}} < \infty,
\end{align*}
which completes the proof, by the Borel-Cantelli lemma.
\end{proof}

A similar simpler proof also gives an upper bound for $\log (-B_t)/\log t$ under the law $\P^\phi$.
\begin{lemma}
\label{lm:BeTypeFluc:ub}
Given $\phi$ a function satisfying \eqref{eqn:hypPhi}, we have
\[
  \limsup_{t \to \infty} \frac{\log(-B_t) }{\log t} \leq \frac{1}{2} \quad \text{a.s.}
\]
\end{lemma}

\begin{proof}
Let $\alpha  > 1/2$. We observe that for all $n \in \N$ we have
\begin{align*}
  \P^\phi( \inf_{s \in [n, n +1]} B_s \leq - n^\alpha)
  &\leq \frac{1}{R^\phi(0,0)} \E(R^\phi(B_{n+1},n+1)\ind{\inf_{s \in [n,n+1]} B_s \leq- n^\alpha})\\
  &\leq C e^{- c n^{2 \alpha - 1}},
\end{align*}
using that $R^\phi(x,n+1)$ grows at most linearly in $-x$, and the Gaussian concentration of $\inf_{s \in [n,n+1]} B_s$. As a result, by the Borel-Cantelli lemma we conclude that
\[
  \limsup_{t \to \infty} \frac{\log (-B_t)}{\log t} \leq \alpha \quad \text{a.s.}
\]
We complete the proof by letting $\alpha \to 1/2$.
\end{proof}

The proof of Lemma \ref{lm:Bessel_type_fluctuations} is then a combination of Lemmas \ref{lm:BeTypeFluc:lb} and \ref{lm:BeTypeFluc:ub}.

\subsection{Linear growth}

In this section we prove the key property of $R^\phi$: the function grows linearly in $-x$ uniformly in $t$. We begin with the following lower bound on $R^\phi$, which is a straightforward consequence of the definition in Theorem \ref{thm:Novikov_existence}.
\begin{lemma}
Let $\phi$ be a function satisfying \eqref{eqn:hypPhi}, then for all $t \geq 0$ and $x \leq \phi(t)$,
\[
  R^\phi(x,t) \geq \phi(t)-x.
\]
\end{lemma}

\begin{proof}
Recall that for all $s \geq 0$ we have $\phi_t(s) = \phi(t+s) \geq \phi(t)$, as $\phi$ is increasing. Therefore, by Theorem \ref{thm:Novikov_existence} we have $\tau_\phi < \infty$ a.s. and 
\[
  R^\phi(x,t) = \E B_{\tau_{\phi_t-x}} \geq \phi(t)-x,
\]
completing the proof.
\end{proof}

To obtain a uniform upper bound on $R^\phi$, we need to add an assumption on the growth rate of the derivative of $\phi$.
\begin{lemma}
\label{lm:linearGrowth}
Let $\phi$ be a function satisfying \eqref{eqn:hypPhi}, and assume additionally that $\phi'(t) = o(t^{-1/2-\epsilon})$ for some $\epsilon > 0$. Then for all $\delta > 0$ and $D > 0$ there exists $t_0 > 0$ such that
\begin{equation}
  \label{eqn:aimRasymptotics}
  \forall t \geq t_0, \ \forall x \in [\phi(t) - D t,\phi(t) - D t_0], \ R^\phi(x,t) \leq (\phi(t)-x) (1 + \delta). 
\end{equation}
\label{lm:R_asymptotics}
\end{lemma}

\begin{proof}
Observe that by the assumption on the function $\phi$, there exists $\gamma < 1/2$ and $A > 0$ such that for all $t \geq 0$ we have $0 \leq \phi'(t) \leq A \gamma t^{\gamma-1}$. By integration we immediately obtain that for all $s,t \geq 0$
\[
  \phi(t+s) - \phi(t) \leq \psi(t+s) - \psi(t),
\]
where we have set $\psi(t) = A t^\gamma$. It is then straightforward to note that for all $s,t \geq 0$ and $x \leq \phi(t)$
\[
  \P_{x}\left( B_u \leq \phi(t+u), u \leq s\right) \leq \P_{x}\left( B_u \leq \psi(t+u)- \psi(t) + \phi(t), u \leq s \right).
\]
As a result, by Theorem \ref{thm:Novikov_existence} and using that $R^\phi(x,t) = 0$ for $x \geq \phi(t)$, we obtain that
\begin{align}
R^\phi(x,t) \leq R^\psi(x+\psi(t)-\phi(t),t)
\label{eq:R_phi_psi}
\end{align}
for all $x \in \R$ and $t \geq 0$. Therefore, we shall work with $R^\psi$ which will simplify some arguments, and use \eqref{eq:R_phi_psi} to prove \eqref{eqn:aimRasymptotics}.

For $t,x \geq 0$ set
\[
  S^\psi(x,t) := R^\psi(\psi(t)-x,t)  - x = \E_{-x} (B_{\tau_{\psi_t-\psi(t)}}).
\]
Observe that as $\psi$ is concave, for all $s \geq 0$ we have that $\psi_t(s) - \psi(t)$ is decreasing with $t$. Therefore $t \mapsto S^\psi(x,t)$ is decreasing, hence for all $D > 0$ one has $S^\psi(x,t) \leq S^\psi(x,x/D)$ as long as $x \leq Dt$. We shall show that for any $D > 0$ we have $S^\psi(Dt,t) / t \to 0$ as $t \to \infty$.

Fix $D > 0$. For all $\lambda, t > 0$ define 
\[
\psi^\lambda(t) := \frac{1}{\lambda} (\psi(\lambda + \lambda^2 t) - \psi(\lambda)),
\]
and observe that by the scaling property of the Brownian motion we have
\begin{equation}
  \frac{S^{\psi}(\lambda D,\lambda)}{\lambda} = \frac{1}{\lambda} \E_{-\lambda D} (B_{\tau_{\psi_\lambda - \psi(\lambda)}})
  = \E_{-D}( B_{\tau_{\psi^\lambda}}).
  \label{eq:S_lemma_iv}
\end{equation}
Observe that $(\psi^\lambda, \lambda > 1)$ decreases to $0$ as $\lambda \to \infty$. We can also note that the convergence is monotone outside of a compact set. Indeed, for all $u > 0$,
\begin{align*}
  \frac{1}{A}\frac{\dd \psi^\lambda(u)}{\dd \lambda}
  &= \frac{\dd }{\dd \lambda}\left( \frac{1}{\lambda} \left((\lambda + u \lambda^2)^\gamma - \lambda^\gamma \right)\right)\\
  &= \frac{1}{\lambda^2} \left( (1 - \gamma) \lambda^\gamma - (1 - 2 \gamma)(\lambda + u \lambda^2)^{\gamma} - \gamma \lambda (\lambda + u \lambda^2)^{\gamma-1} \right).
\end{align*}
In particular, it appears there exists $\lambda_0>0$ such that for all $u > 1$ and $\lambda > \lambda_0$ we have that $\frac{\dd \psi^\lambda(u)}{\dd \lambda}< 0$. Therefore, setting $\bar{\psi}^\lambda(u) := \psi^\lambda(u\vee 1)$, we have
\[
  0\leq \E_{-D}( B_{\tau_{\psi^\lambda}}) \leq \E_{-D}(B_{\tau_{\bar{\psi}^\lambda}}) \to 0 \quad \text{ as } \lambda \to \infty,
\]
by the monotone convergence theorem, using that $\bar{\psi}^\lambda$ decreases to $0$ when $\lambda \to \infty$. Therefore, \eqref{eq:S_lemma_iv} yields
\[
  \lim_{t\to \infty} S^{\psi}(D t ,t)/t = 0.
\]

Choose $\delta > 0$. There exists $t_0>0$ such that for all $t > t_0$ we have $S^\psi(Dt,t) \leq \delta D t$. Then, recalling \eqref{eq:R_phi_psi}, for all $D t_0 \leq y \leq D t$ we have
\[
  R^\phi(\phi(t) - y , t) \leq R^\psi(\psi(t) - y , t) = y + S^\psi(y,t) \leq y + S^\psi(y,y/D)\leq (1 + \delta)y,
\]
which, setting $x := \phi(t) - y$, completes the proof.
\end{proof}

\section{Multidimensional Branching Brownian Motion and uniformly integrable approximations of the martingale}
\label{sec:BBM_and_shaved_martingale}

In this section we prove Theorem \ref{thm:main}, showing that the derivative martingale almost surely converges in almost every direction simultaneously. As we mentioned in the introduction, the techniques are based on a shaving argument: removing all particles that travel too far away from the origin, and therefore carry most of the fluctuations of $Z$. It turns the derivative martingale into a uniformly integrable martingale. We use here the results obtained in the previous section to construct a shaving argument with a function satisfying \eqref{eqn:hypPhi}.

Before moving to the multidimensional setting, we are going to define the  martingale $Z^\phi$ in dimension $1$, that will serve as a uniformly integrable approximation of the derivative martingale $Z$. To be precise, set
\begin{align*}
\mathcal{N}^{\phi}_t := \{ j \in \mathcal{N}_t : X_s(j) \leq \sqrt{2} s + \phi(s) , s \leq t\}.
\end{align*}
The  martingale $Z^\phi$ is then defined in the following way.
\begin{proposition}
	Let $\phi$ be a function satisfying \eqref{eqn:hypPhi}. We set $R^\phi$ as in \eqref{eq:R_def}. Then the process defined for all $t \geq 0$ by
	\begin{align*}
	Z_t^{\phi} := \sum_{j \in \mathcal{N}_t^\phi} R^{\phi}(X_t(j) - \sqrt{2}t, t)e^{\sqrt{2}( X_t(j)  - \sqrt{2}t)}
	\end{align*}
	is a non-negative martingale with mean $R^\phi(0,0)$.
	\label{prop:non_linear_shaving_definition}
\end{proposition}

\begin{proof}
We first note that by definition, $\E Z_0^\phi = R^\phi(0,0)$, and that for all $t,x$, we have $R^\phi(x,t) \geq 0$. We thus only need to check that $Z_t^{\phi}$ is a martingale. By the branching property, for all $s,t \geq 0$ we have
\[
  \E(Z_{t+s}^\phi \ | \  \mathcal{F}_t) = \sum_{j \in \mathcal{N}_t^\phi} G_s(X_t(j)),
\]
where we have set
\begin{align*}
  & G_s(x) \\
  := {} & \E\left( \sum_{j \in \mathcal{N}_s^{\sqrt{2}t + \phi_t - x}} R^\phi(X_s(j)+x - \sqrt{2}(t+s),t+s) e^{\sqrt{2}( X_s(j) + x  - \sqrt{2}(t+s))}\right)\\
  = {} & e^{\sqrt{2} (x - \sqrt{2} t)} \E\Bigg(\sum_{j \in \mathcal{N}_s^{\sqrt{2}t + \phi_t - x}} R^\phi(X_s(j)+x  - \sqrt{2}(t+s) ,t+s) \\
  & \qquad \qquad \qquad \qquad \qquad \qquad \qquad \qquad \qquad \qquad \cdot e^{\sqrt{2}( X_s(j)  - \sqrt{2}s)} \Bigg)\\
  = {} & e^{\sqrt{2} (x - \sqrt{2} t)} e^{s} \E \Big( R^\phi(B_s + x - \sqrt{2}(t+s), t+s) \\
  & \qquad \qquad \qquad \qquad \qquad \qquad \qquad \cdot e^{\sqrt{2} B_s - 2 s} \ind{\forall u \leq s, B_u + x - \sqrt{2}t \leq \sqrt{2} u + \phi_t(u)} \Big),
\end{align*}
by the many-to-one lemma (a corollary of Lemma 1 in \cite{HR2017}). Thus by Lemma~\ref{lm:1dim_change} we obtain
\[
  G_s(x) = e^{\sqrt{2} (x - \sqrt{2} t)} R^\phi(x - \sqrt{2}t,t),
\]
from which we deduce that $\E(Z_{t+s}^\phi \ | \ \mathcal{F}_t) = Z_t^\phi$ a.s., completing the proof.
\end{proof}

\subsection{Construction of \texorpdfstring{$(Z_t^{\phi}(\theta), t\ge 0)$}{the killed derivative martingale}: radial shaving}

We may now turn to our main object of interest : the $d$-dimensional branching Brownian motion $X_t=(X_t(i), i\in \mathcal{N}_t)$. Recall that this is a $d$-dimensional branching particle system in which particles move according to i.i.d. Brownian motions and split into two at rate one. For a direction $\theta \in \S^{d-1}$ recall that
\begin{align*}
Z_t(\theta) = \sum_{j \in \mathcal{N}_t} (\sqrt{2} t - \sp{X_t(j)}{\theta}) e^{\sqrt{2} (\sp{X_t(j)}{\theta} - \sqrt{2} t)}.
\end{align*}

We now introduce the shaved martingale $Z^\phi$, where the shaving is done along a curve $\phi$ satisfying \eqref{eqn:hypPhi}. Set $\mathcal{N}^{\phi,\theta}_t = \{ j \in \mathcal{N}_t : \sp{X_s(j)}{\theta} \leq \sqrt{2} s + \phi(s) , s \leq t\}$ for $t \geq 0$ and $\phi \in \S^{d-1}$.

We now set
\begin{align}
Z_t^{\phi}(\theta) := {} & \sum_{j \in \mathcal{N}^{\phi, \theta}_t} R^{\phi}(\sp{X_t(j)}{\theta} - \sqrt{2}t, t)e^{\sqrt{2}( \sp{X_t(j)}{\theta} - \sqrt{2}t)}.
\label{eq:non_linear_shaving_theta}
\end{align}
The function $\phi$ will be chosen to grow fast enough to guarantee that
\[
  \lim_{A \to \infty} \P\left( \forall t \geq 0,   \cap_{\theta \in \S^{d-1}} \mathcal{N}^{\phi \vee A,\theta}_t = \mathcal{N}_t\right) = 1.
\]
We will show in Section \ref{subsec:ddim_displacement} that choosing $\phi$ growing faster than $\frac{d-1}{2\sqrt{2}} \log t$ as $t \to \infty$ is enough.

In Section \ref{subsec:uniform_integrability} we prove (using classical spinal decomposition techniques along the lines of \cite{Kyprianou2004} and \cite{Roberts2010}) that for all measurable bounded functions $f$ the process $(\langle Z^\phi_t,f \rangle, t \geq 0)$ is a uniformly integrable martingale. We then use convergence of these martingales in Section \ref{subsec:simultaneous_limit} to show that $\lim_{t \to \infty} Z^\phi_t(\theta)$ exists for almost all $\theta \in \S^{d-1}$ almost surely.  Finally, we complete the proof of Theorem \ref{thm:main} by using that with high probability $Z$ and $Z^\phi$ coincide asymptotically as $t \to \infty$, for which we shall apply Lemma \ref{lm:linearGrowth}.

\subsection{Bounds on the maximal displacement of the BBM}
\label{subsec:ddim_displacement}

We prove here that with high probability all particles in the multidimensional BBM are at all times $t$ within a ball of radius $\sqrt{2} t + \frac{d-1}{2\sqrt{2}} \log (t+1) + A$. First, recall the following lemma due to Mallein:
\begin{lemma}[\cite{Mallein2015}, Lemma 3.1]
Let
\begin{align*}
r_s^{t,y} := \sqrt{2}s + \frac{d-1}{2\sqrt{2}}\log(s+y) - \frac{3}{2\sqrt{2}} \log \frac{t+1}{t-s+1} + y.
\end{align*}
Then there exists $C > 0$ such that for any $t \geq 1$ and $y \in [1, \sqrt{t}]$
\begin{align*}
\Prob\left(\exists j 
\in \mathcal{N}_t, \exists s \leq t : ||X_s(j)|| \geq r_s^{t,y}\right) \leq Cye^{-\sqrt{2}y}.
\end{align*}
\label{fact:Mallein3.1}
\end{lemma} 

We use Lemma \ref{fact:Mallein3.1} to prove the following result.
\begin{lemma}
Let $\tilde r(s) := \sqrt{2}s + \frac{d-1}{2\sqrt{2}} \log (1 + s)$. For any $\epsilon > 0$ there exists $C_\epsilon$ such that
\begin{align*}
\Prob\left(\exists t \geq 0, \exists j 
\in \mathcal{N}_t : ||X_t(j)|| \geq \tilde r(t)  + C_\epsilon\right) \leq \epsilon.
\end{align*}
\label{lm:push_barrier}
\end{lemma}

\begin{proof}
Observe first that by Lemma \ref{fact:Mallein3.1}, for any $y > 0$ and $t \geq 0$, we have
\begin{multline*}
\P\left(\exists s \leq t, \exists j \in \mathcal{N}_s : ||X_s(j)|| \geq \sqrt{2}s + \frac{d-1}{2\sqrt{2}}\log(s+y) + y \right)\\
  \leq \P\left( \exists s \leq t, \exists j \in \mathcal{N}_s : ||X_s(j)|| \geq r_s^{t,y} \right)
  \leq Cye^{-\sqrt{2}y}.
\end{multline*}
Hence, choosing $y$ large enough such that $C y e^{-\sqrt{2}y} < \epsilon$ and letting $t \to \infty$, we deduce that
\[
  \P( \exists s \geq 0, j \in \mathcal{N}_s :  ||X_s(j)|| \geq \sqrt{2}s + \frac{d-1}{2\sqrt{2}}\log(s+y) + y ) \leq \epsilon.
\]
To complete the proof, it is therefore enough to choose $C_\epsilon$ as
\[
  \sup_{t \geq 0}  \frac{d-1}{2\sqrt{2}}\log(t+y) + y - \frac{d-1}{2\sqrt{2}} \log (t + 1) <\infty. \qedhere
\]
\end{proof}

\subsection{Uniform integrability of \texorpdfstring{$(Z_t^\phi,t\ge 0)$}{the killed derivative martingale}}
\label{subsec:uniform_integrability}
Let $f$ be a non-negative function such that $\int_{\S^{d-1}} f(\theta) \sigma(\dd \theta) = 1$. By Fubini's theorem, it is a straightforward calculation to verify that the process defined by
\[
  \langle Z_t^\phi,f \rangle = \int_{\S^{d-1}} Z_t^\phi(\theta) f(\theta) \sigma(\dd \theta)
\]
is a non-negative martingale. To prove its uniform integrability we use a spinal decomposition method. This technique, pioneered by Lyons, Pemantle and Peres \cite{LPP1995} for studying Galton-Watson processes, and adapted by Lyons \cite{Lyons1997} to spatial branching settings, consists in an alternative description of the law of the branching Brownian motion biased by the martingale $\langle Z_t^\phi,f \rangle$. More precisely, we define
\[
  \left. \frac{\dd \P^f}{\dd \P} \right|_{\mathcal{G}_t} := R^\phi(0,0)^{-1}\langle Z_t^\phi, f \rangle.
\]
The spinal decomposition consists in a construction of the BBM under the law $\P^f$, where a distinguished particle, called the spine, moves and reproduces differently to typical BBM particles. The offspring of that spine particle then start independent copies of the original BBM with law $\P$, from their birth time and position.

Before presenting the spinal decomposition for the branching Brownian motion, we introduce the law of the multi-dimensional Brownian motion biased by a martingale similar to the one introduced in Lemma \ref{lm:1dim_change}. This will allow us to describe the trajectory of the spine under the biased law $\P^f$.

Let $B$ be a Brownian motion in $\R^d$. For all $\theta \in \S^{d-1}$ we define a non-negative martingale $(V_t(\theta), t\geq 0)$ as
\[
  V_t(\theta) := \frac{R^\phi(\sp{B_t}{\theta} - \sqrt{2}t,t)}{R^\phi(0,0)} \ind{\tau_\phi(\theta) > t} e^{\sqrt{2}\sp{B_t}{\theta} - t},
\]
where $\tau_\phi(\theta) :=  \inf \{u > 0 \ : B_u \cdot \theta  - \sqrt{2}u \geq \phi(u) \}$. Writing $B^{(1)} = \sp{B}{\theta}$ and $B^{(2)}$ for the projection of $B$ on $\theta^\perp$, we note that these are two independent Brownian motions. Applying Lemma \ref{lm:1dim_change} to $B^{(1)}$, we deduce that under the law defined as
\[
  \left. \frac{\dd \P^{V(\theta)}}{\dd \P} \right|_{\mathcal{G}_t} := V_t(\theta)
\]
the process $B$ is a $d$-dimensional Brownian motion with drift $\sqrt{2} \theta$, conditioned on $\sp{B_t}{\theta} \leq \sqrt{2} t + \phi(t)$ for all $t \geq 0$ (in the sense of Doob).

The key point of Theorem \ref{thm:main} is to consider several directions at the same time. To do so, we will consider integrated versions of the martingale $V(\theta)$. Given $f$ a non-negative function satisfying $\int_{\S^{d-1}} f(\theta) \sigma(\dd \theta) = 1$, we set
\begin{align*}
U_t := \langle V_t, f \rangle
\end{align*}
and we define the measure $\P^U$ by 
\begin{align*}
\left. \frac{\dd\P^U}{\dd\P}\right|_{\mathcal{G}_t} := U_t.
\end{align*}

\begin{lemma}
Let $f$ be a non-negative function with $\int_{\S^{d-1}} f(\theta) \sigma(\dd \theta) = 1$, then the process $U$ is a non-negative martingale. Moreover, setting $\theta_0$ a random variable in $\S^{d-1}$ with law $f(\theta) \sigma(\dd \theta)$ and writing $(B_t)$ for a process with law $\P^{V(\theta_0)}$  conditionally on $\theta_0$, the process $(B_t, t \geq 0)$ has law $\P^U$.
\end{lemma}

\begin{proof}
The process $U$ is a martingale using Fubini's theorem. Additionally, for all $t \geq 0$ and $G \in \mathcal{G}_t$ we have
\begin{align*}
  \P^U(G) = \int_G \langle V_t, f \rangle \dd \P = \left\langle \int_G V_t \dd \P, f \right\rangle = \int_{\S^{d-1}} \P^{V(\theta)}(G) f(\theta) \dd \theta,
\end{align*}
which justifies the description of $B$ under the law $\P^U$.
\end{proof}

Observe that one can decompose
\begin{align*}
\langle Z^\phi_t, f \rangle =  R^\phi(0,0) \sum_{j \in \mathcal{N}_t} U_t(j) e^{-t},
\end{align*}
where
\begin{align*}
U_t(j) := \left\langle \frac{R^\phi(\sp{X_t(j)}{\theta} - \sqrt{2}t,t)}{R^\phi(0,0)} \ind{\forall u < t, \sp{X_u(j)}{\theta} -  \sqrt{2}t < \phi(u)}e^{\sqrt{2}\sp{X_t(j)}{\theta} - t}, f \right\rangle.
\end{align*}
Thanks to this decomposition we can describe the BBM under the law $\P^f$ in terms of a spinal decomposition, which follows e.g. from  \cite[Lemma 6.7]{HH2006}.
\begin{lemma}
Let $f$ be a non-negative function with $\int_{\S^{d-1}} f(\theta) \sigma(\dd \theta) = 1$. The law of the BBM under $\P^f$ can be constructed as follows
\begin{enumerate}
	\item we pick a direction $\theta_0$ according to a random variable on $\S^{d-1}$ with density $f(\theta)\sigma(\dd \theta)$;
	\item conditionally on this direction we sample a trajectory $(\Xi_t)$ with law $\P^{V(\theta_0)}$ that will be followed by the spine particle;
	\item the spine particle creates offspring at rate $2$;
	\item every child of the spine then starts an independent standard BBM with law $\P$.
\end{enumerate}
\label{lm:spinaldecomposition}
\end{lemma}

An analogous decomposition in dimension one was given in \cite{CR1988} or in \cite{Kyprianou2004}. We are now ready to present the key lemma that states the uniform integrability of $Z_t^\phi$.
\begin{lemma}
Let $\phi$ be a function satisfying \eqref{eqn:hypPhi}. For any bounded measurable function $f$ the martingale $\left(\langle Z^\phi_t, f \rangle\right)_{t \geq 0}$ is uniformly integrable.
\label{lm:uniform_integrability_shaved}
\end{lemma}
Before we present the proof of Lemma \ref{lm:uniform_integrability_shaved}, note that applying it in dimension one with the binary function $f$ (i.e. $f(-1) = 0$ and $f(1) = 1$) we obtain the following corollary.
\begin{corollary}
	For any $\theta \in \S^{d-1}$, $\left(Z_t^{\phi}(\theta)\right)_{t \geq 0}$ is a uniformly integrable martingale.
\end{corollary}

\begin{proof}[Proof of Lemma \ref{lm:uniform_integrability_shaved}.]
Note first that without loss of generality we may assume that $f \geq 0$ and that $\int_{\S^{d-1}} f(\theta) \sigma(\dd \theta) = 1$, as otherwise we may write $f$ as a linear combination of functions satisfying these assumptions and consider each of these functions separately. 

Set $\mathcal{Z} := \limsup_{t \to \infty} \langle Z^\phi_t, f \rangle$ (which is also equal to $\lim_{t \to \infty} \langle Z^\phi_t, f \rangle$ $\P$-a.s. because $\langle Z^\phi_t, f \rangle$ is a non-negative martingale). Recall the following measure theoretic dichotomy (see e.g. Theorem 5.3.3. in \cite{Durrett2010}):
\begin{theorem}
	Let $(\mathcal{F}_n)$ be a filtration, and let $\mathcal{F}_\infty$ be the smallest $\sigma$-field containing all $\mathcal{F}_n$. Let $\Prob, \Q$ be two probability measures on $(\Omega, \mathcal{F}_\infty)$. Assume that for any $n$, $\Q_{|\mathcal{F}_n} \ll \Prob_{|\mathcal{F}_n}$ and let $X_n := \frac{\dd\Q_{|\mathcal{F}_n}}{\dd\Prob_{|\mathcal{F}_n}}$ and $X := \limsup_{n \to \infty} X_n$ which is $\Prob$-a.s. finite. Then
	\begin{align*}
	\Q(A) = \E(X \ind{A}) + \Q(A \cap \{X = \infty\}), \quad \forall A \in \mathcal{F}_\infty.
	\end{align*}
	\label{thm:dichotomy}
\end{theorem}
From Theorem \ref{thm:dichotomy} we obtain that
\begin{align*}
\P^f\left(\frac{\mathcal{Z} }{R^\phi(0,0)}< \infty \right) = 1 \quad \iff \quad \int \frac{\mathcal{Z} }{R^\phi(0,0)} \dd \P = 1,
\end{align*}
thus instead of proving that $\E \mathcal{Z} = 1$, we shall prove that under $\P^f$, $\mathcal{Z}$ is almost surely finite. To show that, we are going to use the spinal decomposition from Lemma \ref{lm:spinaldecomposition}.

Let $\mathcal{F}_\infty$ be the filtration generated by the movement and the branching of the spine $\Xi$, and $\mathcal{B}_t$ be the set of branching times of the spine until time~$t$. From the decomposition mentioned above and the martingale property from Proposition \ref{prop:non_linear_shaving_definition} we see that
\begin{align*}
\P^f[\langle Z^\phi_t, f \rangle|\mathcal{F}_\infty] = {} &  \left\langle \sum_{s \in \mathcal{B}_t} R^\phi(\sp{\Xi_s}{\theta} - \sqrt{2}s ,s) e^{\sqrt{2}(\sp{\Xi_s}{\theta} - \sqrt{2}s)}, f \right\rangle\\
{} & +  \left\langle R^\phi(\sp{\Xi_t}{\theta} - \sqrt{2}t ,t) e^{\sqrt{2}(\sp{\Xi_t}{\theta} - \sqrt{2}t)}, f \right\rangle.
\end{align*}
To complete the proof it is enough to show that 
\begin{align}
\limsup_{t \to \infty} \P^f[\langle Z^\phi_t, f \rangle \ |\ \mathcal{F}_\infty]  < \infty, 
\label{eq:Z_under_Q}
\end{align}
as by Fatou's lemma we have
\begin{align*}
& \P^f[ \liminf_{t \to \infty} \langle Z^\phi_t, f \rangle \ | \ \mathcal{F}_\infty]\\
\leq {} & \liminf_{t \to \infty} \P^f[\langle Z^\phi_t, f \rangle \ | \ \mathcal{F}_\infty]\\
\leq {} & \limsup_{t \to \infty} \P^f[\langle Z^\phi_t, f \rangle \ | \ \mathcal{F}_\infty]  < \infty, 
\end{align*}
which implies that  $\P^f$-a.s., $\liminf_{s \to \infty} \langle Z^\phi_t, f \rangle < \infty$. Recalling the definition of $\P^f$, $(\langle Z^\phi_t, f \rangle)^{-1}$ is a non-negative $\P^f$-supermartingale, hence it converges to a finite limit $\P^f$ almost surely. This implies that $\P^f$ almost surely
\begin{align*}
\liminf_{s \to \infty} \langle Z^\phi_t, f \rangle = \limsup_{s \to \infty} \langle Z^\phi_t, f \rangle,
\end{align*}
from which we would deduce that $\lim_{t \to \infty} \langle Z^\phi_t, f \rangle < \infty$ $\P^f$-a.s.

It remains to show (\ref{eq:Z_under_Q}). We first upper bound $||\Xi_t|| = \sup_{\theta} \sp{\Xi_t}{\theta}$. Fix the direction $\theta_0$ in which the movement of the spine is altered. Observe that we can decompose the spine as $\Xi_t = \xi_t \theta_0 + Y_t$ where $\xi_t$ and $Y_t$ are independent processes such that $\xi_t$ is a Brownian motion with drift $\sqrt{2}t$ conditioned on never hitting $\sqrt{2}t + \phi(t)$ and $Y_t$ is a $(d-1)$-dimensional Brownian motion living in the space $\theta_0^{\perp}$. Thus
\begin{align*}
||\Xi_t|| = \sqrt{|\xi_t|^2+||Y_t||^2}.
\end{align*}

By Lemma \ref{lm:Bessel_type_fluctuations} almost surely for any $\delta>0$ there exist $C_1, t_0$ such that for all $t \geq t_0$, $|\xi_t| \leq \sqrt{2}t - C_1 t^{1/2 -\delta}$. Similarly, by e.g. the law of the iterated logarithm, for any $\delta' >0$ there exists $C_2$ such that  up to enlarging $t_0$, for all $t \geq t_0$, $||Y_t|| \leq C_2 t^{1/2 + \delta'}$. Choose $\delta, \delta'$ such that $\delta + 2 \delta' < 1/2$, then for $t$ large enough,
\begin{align}
||\Xi_t|| \leq {} & \sqrt{2t^2 + C_1^2t^{1-2\delta}-2\sqrt{2}C_1 t^{3/2-\delta}+C_2^2 t^{1+2\delta'}} \nonumber\\
\leq {} & \sqrt{2t^2 + (C_1/2)^2t^{1-2\delta}-2\sqrt{2}(C_1/2)t^{3/2-\delta}} \nonumber \\
= {} & \sqrt{2}t-C_1/2t^{1/2-\delta}.
\label{eq:spine_modulus}
\end{align}
Let $C_f = \sup_{\S^{d-1}} f(\theta)$. By Lemma \ref{lm:R_def} we know that for some $C \geq 0$, $R^\phi(x,t) \leq C(1+|x|+\phi(t))$ for all $x \in \R, t \geq 0$, thus since the spine particle has zero contribution in the limit, 
\begin{align*}
& \limsup_{t \to \infty} \P^f[\langle Z^\phi_t, f \rangle \ | \ \mathcal{F}_\infty]\\
\leq {} & C \left\langle \sum_{s \in \mathcal{B}_\infty} (1 + |\sqrt{2}s -  \sp{\Xi_s}{\theta}| + \phi(s) ) e^{\sqrt{2}(\sp{\Xi_s}{\theta} - \sqrt{2}s)}, C_f \right\rangle \\
\leq {} & P_{d-1} C C_f\sum_{s \in \mathcal{B}_\infty}  (1 + \sqrt{2}s + ||\Xi_s|| + \phi(s) ) e^{\sqrt{2}(||\Xi_s|| - \sqrt{2}s)},
\end{align*}
where $P_{d-1}$ is the surface area of a $d$ dimensional sphere. Combining it with (\ref{eq:spine_modulus}) we obtain that almost surely there exists a constant $C_1$ such that
\begin{align*}
\limsup_{t \to \infty} \P^f[\langle Z^\phi_t, f \rangle \ | \ \mathcal{F}_\infty] \leq P_{d-1} C C_f\sum_{s \in \mathcal{B}_\infty}  (1 + 2\sqrt{2}s + \phi(s)) e^{-s^{1/2-\delta}C_1/2},
\end{align*}
which is almost surely finite, as $\mathcal{B}$ is a Poisson point process with intensity~$2$. The proof is now complete.
\end{proof}

\subsection{Simultaneous limits  {on the sphere}}
\label{subsec:simultaneous_limit}

The main aim of this section is the proof of the following proposition, which shows that $(Z_t(\theta))$ converges a.s. both on a random set of full Lebesgue measure, and as a random measure.
\begin{proposition}
	Let $\phi$ be a function satisfying \eqref{eqn:hypPhi}. Then almost surely there exists $\Theta \subset \S^{d-1}$ of full Lebesgue measure such that for all $\theta \in \Theta$, $Z^\phi_\infty(\theta) := \lim_{t\to\infty} Z^\phi_t(\theta)$ exists, and for any bounded measurable function $f$,
	\begin{align*}
	\lim_{t \to \infty} \langle Z^\phi_t, f \rangle = \langle Z^\phi_\infty, f \rangle \textrm{ a.s.}
	\end{align*}
	Moreover, the limit is almost surely finite.
	\label{prop:shaved_swap_integral_limit}
\end{proposition}

\begin{proof}
Without loss of generality we may and will assume that $f \geq 0$ and $\int_{\S^{d-1}} f(\theta) \sigma(\dd \theta) = 1$. The integrated  martingale $\langle Z^\phi_t, f \rangle$ is non-negative, hence it converges a.s. to some limit, and we set $\mathcal{Z} := \lim_{t \to \infty} \langle Z^\phi_t, f \rangle$. Furthermore, by Lemma \ref{lm:uniform_integrability_shaved} this martingale is uniformly integrable, thus
\begin{align*}
\E \mathcal{Z} = \E \langle Z^\phi_0, f \rangle = R^\phi(0,0).
\end{align*}
We want to show that
\begin{align*}
\mathcal{Z} = \langle \lim_{t\to\infty} Z_t^{\phi}, f \rangle
\end{align*}
but {\it a priori} we don't even know that the right hand side is well defined.

As $Z_t^\phi(\theta) \geq 0$ a.s., we observe that by Fatou's lemma
\begin{align}
  \mathcal{Z} = \liminf_{t \to \infty} \langle Z_t^{\phi}, f \rangle \geq \langle \liminf_{t \to \infty} Z_t^{\phi}, f \rangle.
\label{eq:Z_infty_fubini}
\end{align}
Note that $\liminf_{t \to \infty} Z_t^{\phi}(\theta)$ exists simultaneously for all $\theta \in \S^{d-1}$. 

On the other hand, by the uniform integrability of $\langle Z^\phi_t, f \rangle$ (Lemma \ref{lm:uniform_integrability_shaved}) and Fubini's theorem,
\begin{align*}
\E \mathcal{Z} = \lim_{t \to \infty} \E \langle Z_t^{\phi}, f \rangle = \lim_{t \to \infty} \langle \E Z_t^{\phi}, f \rangle.
\end{align*}
Since the distribution of $Z_t^{\phi}(\theta)$ does not depend on $\theta$ (used in the first equality), and again using the uniform integrability, but of $Z_t^{\phi}(\theta)$, and also Fubini's theorem, we obtain that
\begin{align*}
\lim_{t \to \infty} \langle \E Z_t^{\phi}, f \rangle = {} & \langle \lim_{t \to \infty} \E Z_t^{\phi}, f \rangle 
= {}  \langle \E  \lim_{t \to \infty} Z_t^{\phi}, f \rangle \\
= {} & \langle \E  \liminf_{t \to \infty} Z_t^{\phi}, f \rangle
= {}  \E \langle \liminf_{t \to \infty} Z_t^{\phi}, f \rangle.
\end{align*}
Thus we have shown that 
\begin{align*}
\E\mathcal{Z} = \E \langle \liminf_{t \to \infty} Z_t^{\phi}, f \rangle,
\end{align*}
and recalling (\ref{eq:Z_infty_fubini}) this means that almost surely
\begin{align*}
\mathcal{Z} = \langle \liminf_{t \to \infty} Z_t^{\phi}, f \rangle. 
\end{align*}
By Fubini's theorem
\begin{align*}
\E \langle \limsup_{t \to \infty} Z_t^{\phi}, f \rangle = \langle \E  \lim_{t \to \infty} Z_t^{\phi}, f \rangle = \E \langle \liminf_{t \to \infty} Z_t^{\phi}, f \rangle,
\end{align*}
hence almost surely for almost all $\theta$
\begin{align*}
\limsup_{t \to \infty} Z_t^{\phi}(\theta) = \liminf_{t \to \infty} Z_t^{\phi}(\theta).
\end{align*}
Therefore, almost surely $\lim_{t\to\infty} Z_t^{\phi}(\theta)$ exists  simultaneously for all $\theta$ besides a random set of measure $0$, and 
\begin{align*}
\lim_{t \to \infty} \langle Z_t^{\phi}, f \rangle = \langle \lim_{t \to \infty} Z_t^{\phi}, f \rangle.
\end{align*}
\end{proof}

\subsection{Proof of Theorem \ref{thm:main}}
We start with the following technical lemma: 

\begin{lemma}
Let $\phi$ be a function satisfying \eqref{eqn:hypPhi}. If the function $\phi$ additionally satisfies $\lim_{t \to \infty} \phi(t) - \frac{d-1}{2\sqrt{2}} \log (1 + t) = \infty$ and $\phi'(t) = o(t^{-1/2-\epsilon})$, then for any bounded measurable function $f$,
\begin{multline}
\lim_{t \to \infty} \left\langle \sum_{j \in \mathcal{N}_t^{\phi, \theta}} (\sqrt{2}t - \sp{X_t(j)}{\theta}  + \phi(t))e^{\sqrt{2}(\sp{X_t(j)}{\theta} -\sqrt{2}t)}, f \right\rangle \\
= \left\langle \lim_{t \to \infty} \sum_{j \in \mathcal{N}_t^{\phi, \theta}} (\sqrt{2}t - \sp{X_t(j)}{\theta}  + \phi(t))e^{\sqrt{2}(\sp{X_t(j)}{\theta} -\sqrt{2}t)}, f \right\rangle
\label{eq:Z_replace_R}
\end{multline}
almost surely and the limit is finite with probability one.
\label{lm:Z_replace_R}
\end{lemma}

Note that there are two differences between Lemma~\ref{lm:Z_replace_R} and Theorem~\ref{thm:main}: firstly, we don't take a sum over all particles, and secondly we have an additional term $\phi$ appearing. We solve both of these issues in the remainder of this section.

\begin{proof}[Proof of Lemma \ref{lm:Z_replace_R}]
Recall the definition \eqref{eq:non_linear_shaving_theta}. Since 
\begin{align*}
\lim_{t \to \infty} \phi(t) - \frac{d-1}{2\sqrt{2}} \log (1 + t) = \infty,
\end{align*} from Lemma \ref{lm:push_barrier} we obtain that 
\begin{align*}
\lim_{t\ \to \infty} \inf_{j \in \mathcal{N}_t} \left(\sqrt{2}t - ||X_t(j)|| + \phi(t)\right) = +\infty
\end{align*}
and
\begin{align*}
\limsup_{t \to \infty} \sup_{\theta \in \S^{d-1}, j \in \mathcal{N}_t} \frac{1}{t}(\sqrt{2}t - \sp{X_t(j)}{\theta}) = 2\sqrt{2}
\end{align*}
almost surely. We are now going to make use of the asymptotic behaviour of $R^\phi(x,t)$: we apply Lemma \ref{lm:R_asymptotics} with $D > 2\sqrt{2}$ and an arbitrarily small $\delta$ to obtain that almost surely
\begin{align}
\langle Z^\phi_\infty, f \rangle
= \langle \lim_{t \to \infty} \sum_{j \in \mathcal{N}_t^{\phi, \theta}} (\sqrt{2}t - \sp{X_t(j)}{\theta}  + \phi(t))e^{\sqrt{2}(\sp{X_t(j)}{\theta} -\sqrt{2}t)}, f \rangle.
\label{eq:Z_replace_R_1}
\end{align}
From Proposition \ref{prop:shaved_swap_integral_limit} we know that
\begin{align*}
\langle Z^\phi_\infty, f \rangle = \lim_{t \to \infty} \langle Z^\phi_t, f \rangle.
\end{align*}
and again, applying Lemma \ref{lm:R_asymptotics} with $D > 2\sqrt{2}$ and an arbitrarily small $\delta$, we obtain that 
\begin{align}
\langle Z^\phi_\infty, f \rangle
= \lim_{t \to \infty}  \left\langle \sum_{j \in \mathcal{N}_t^{\phi, \theta}} (\sqrt{2}t - \sp{X_t(j)}{\theta}  + \phi(t))e^{\sqrt{2}(\sp{X_t(j)}{\theta} -\sqrt{2}t)}, f \right\rangle.
\label{eq:Z_replace_R_2}
\end{align}
Combining \eqref{eq:Z_replace_R_1} and \eqref{eq:Z_replace_R_2} completes the proof.
\end{proof}

We now get rid of the term involving $\phi$ in \eqref{eq:Z_replace_R}:
\begin{lemma}
Let $\phi$ be such that $\phi(t) = o(t^{1/2-\epsilon})$ for some $\epsilon > 0$. Then
\begin{align*}
\lim_{t \to \infty} \phi(t) \left\langle \sum_{j \in \mathcal{N}_t} e^{\sqrt{2}(\sp{X_t(j)}{\theta} -\sqrt{2}t)}, 1 \right\rangle = 0
\end{align*}
almost surely.
\label{lm:Z_infty_part_with_phi}
\end{lemma}
\begin{proof}
Without loss of generality assume that $\phi(t)$ is an increasing, concave, $\mathcal{C}^1$-class function such that $\lim_{t \to \infty} \phi(t) - \frac{d-1}{2\sqrt{2}} \log (1 + t) = \infty$ and $\phi'(t) = o(t^{-1/2-\epsilon})$. Set $\psi(t) := t^{1/2-\epsilon/2}$ and observe that by Lemma \ref{lm:push_barrier}, for any $\delta > 0$ we can choose $A_\delta$ such that with probability $1-\delta$ none of the particles ever hit the sphere of an increasing radius $\sqrt{2} t + \frac{d-1}{2\sqrt{2}} \log(1+t) + A_\delta$, thus conditioning on this event
\begin{align}
\begin{split}
& \limsup_{t \to \infty} \ \left\langle \sum_{j \in \mathcal{N}_t^{\psi +A_\delta, \theta}} (\sp{X_t(j)}{\theta} - \sqrt{2}t)\ind{\sp{X_t(j)}{\theta} \geq \sqrt{2}t}e^{\sqrt{2}(\sp{X_t(j)}{\theta} -\sqrt{2}t)}, 1 \right\rangle \\
\leq {} & \limsup_{t \to \infty} \ \left\langle \sum_{j \in \mathcal{N}_t^{\psi +A_\delta, \theta}} (\phi(t)+A_\delta)\ind{\sp{X_t(j)}{\theta} \geq \sqrt{2}t}e^{\sqrt{2}(\sp{X_t(j)}{\theta} -\sqrt{2}t)}, 1 \right\rangle \\
\leq {} & \limsup_{t \to \infty} \frac{\phi(t)+A_\delta}{\psi(t)+A_\delta} \left\langle \sum_{j \in \mathcal{N}_t^{\psi +A_\delta, \theta}} (\psi(t)+A_\delta)e^{\sqrt{2}(\sp{X_t(j)}{\theta} -\sqrt{2}t)}, 1 \right\rangle.
\end{split}
\label{eq:Z_infty_negative_part_bound}
\end{align}
Consider the following decomposition:
\begin{align}
\begin{split}
\sqrt{2}t - \sp{X_t(j)}{\theta}  + \psi(t) + A_\delta = {} & (\sqrt{2}t - \sp{X_t(j)}{\theta}) \ind{\sp{X_t(j)}{\theta} \geq \sqrt{2}t} \\
{} & +(\sqrt{2}t - \sp{X_t(j)}{\theta}) \ind{\sp{X_t(j)}{\theta} \leq \sqrt{2}t}\\
{} & +\psi(t) + A_\delta.
\end{split}
\label{eq:Z_infty_negative_part}
\end{align}

Note that only the first term on the right-hand side of \eqref{eq:Z_infty_negative_part} is negative. Since by Lemma \ref{lm:Z_replace_R}
\begin{align}
\lim_{t \to \infty}  \left\langle \sum_{j \in \mathcal{N}_t^{\psi + A_\delta, \theta}} (\sqrt{2}t - \sp{X_t(j)}{\theta}  + \psi(t) + A_\delta)e^{\sqrt{2}(\sp{X_t(j)}{\theta} -\sqrt{2}t)}, 1 \right\rangle
\label{eq:Z_psi}
\end{align}
exists almost surely, from (\ref{eq:Z_infty_negative_part}), (\ref{eq:Z_infty_negative_part_bound}) and $\lim_{t \to \infty} \frac{\phi(t) + A_\delta}{\psi(t) + A_\delta}=0$ we deduce that the limit
\begin{align*}
\limsup_{t \to \infty} \ \left\langle \sum_{j \in \mathcal{N}_t^{\psi +A_\delta, \theta}} (\psi(t)+A_\delta)e^{\sqrt{2}(\sp{X_t(j)}{\theta} -\sqrt{2}t)}, 1 \right\rangle
\end{align*}
is finite with probability $1-\delta$: if it wasn't finite with probability larger than~$\delta$, then by \eqref{eq:Z_infty_negative_part_bound} and \eqref{eq:Z_infty_negative_part}, with positive probability \eqref{eq:Z_psi} would diverge to infinity, as its negative part is negligible in comparison to the positive one. 

Since $\lim_{t \to \infty} \frac{\phi(t)}{\psi(t)}=0$, this implies further that with probability $1-\delta$
\begin{align*}
\lim_{t \to \infty} (\phi(t)+A_\delta) \left\langle \sum_{j \in \mathcal{N}_t^{\phi + A_\delta, \theta}} e^{\sqrt{2}(\sp{X_t(j)}{\theta} -\sqrt{2}t)}, 1 \right\rangle = 0.
\end{align*}
Taking $A_\delta$ arbitrarily large completes the proof.
\end{proof}

We are now ready to present the last step of the proof of Theorem \ref{thm:main}. Recalling Lemma \ref{lm:Z_replace_R} we show that in fact we can sum over all the particles and we can still swap integration with taking the limit. As was mentioned before, this is the step where we consider a sequence of functions $\phi \vee A$ for $A \in \mathbb{N}$.

\begin{proof}[Proof of Theorem \ref{thm:main}]
Set $\phi = t^{1/2-\epsilon}$ for some $\epsilon \in (0,1)$. By combining Lemma~\ref{lm:Z_replace_R} with Lemma \ref{lm:Z_infty_part_with_phi} we obtain that almost surely for all $A \in \mathbb{N}$
\begin{multline}
\lim_{t \to \infty} \left\langle \sum_{j \in \mathcal{N}_t^{\phi \vee A, \theta}} (\sqrt{2}t - \sp{X_t(j)}{\theta})e^{\sqrt{2}(\sp{X_t(j)}{\theta} -\sqrt{2}t)}, f \right\rangle\\
=\left\langle \lim_{t \to \infty} \sum_{j \in \mathcal{N}_t^{\phi \vee A, \theta}} (\sqrt{2}t - \sp{X_t(j)}{\theta})e^{\sqrt{2}(\sp{X_t(j)}{\theta} -\sqrt{2}t)}, f \right\rangle.
\label{eq:main_proof_1}
\end{multline}

By Lemma \ref{lm:push_barrier} for any $\delta$ we can choose $A_\delta$ such that the event defined by $B_\delta := \{\forall s > 0, u \in \mathcal{N}_s: ||X_s(j)|| \leq \sqrt{2}s + \phi (s) \vee A_\delta\}$ happens with probability  $\Prob(B_\delta) \geq 1-\delta$. Therefore, conditioning on $B_\delta$ and taking $A \geq A_\delta$ in \eqref{eq:main_proof_1}, we obtain that
\begin{multline}
\lim_{t \to \infty} \left\langle \sum_{j \in \mathcal{N}_t} (\sqrt{2}t - \sp{X_t(j)}{\theta})e^{\sqrt{2}(\sp{X_t(j)}{\theta} -\sqrt{2}t)}, f \right\rangle\\
=\left\langle \lim_{t \to \infty} \sum_{j \in \mathcal{N}_t} (\sqrt{2}t - \sp{X_t(j)}{\theta})e^{\sqrt{2}(\sp{X_t(j)}{\theta} -\sqrt{2}t)}, f \right\rangle.
\label{eq:main_proof_2}
\end{multline} 
holds almost surely on $B_\delta$. Taking $\delta$ arbitrarily small we conclude that \eqref{eq:main_proof_2} holds with probability one, which proves \eqref{eq:thm_main}.

Finally, to show that $\langle Z_\infty(\theta),1 \rangle> 0$ we observe that by Fubini's theorem
\begin{align*}
0 = \int_{\S^{d-1}} \Prob(Z_\infty(\theta) = 0) \sigma(\dd \theta) = \E \left[\int_{\S^{d-1}} \ind{Z_\infty(\theta) = 0} \sigma(\dd \theta) \right],
\end{align*}
which completes the proof.
\end{proof}

\section{Direction of the largest displacement in dimension one}
\label{sec:1dim}

In this section we prove Theorem \ref{thm:main1d} but we start by showing how Corollary \ref{cor:1d} follows from Theorem \ref{thm:main1d}. Set
\begin{align*}
  G^+_t &:= \sqrt{2}(M^+_t-m_t - \tfrac{\sqrt{2}}{2} \log Z_\infty)\\
  \text{and} \quad G^-_t &:= \sqrt{2}(- M^-_t-m_t - \tfrac{\sqrt{2}}{2} \log Z^-_\infty).
\end{align*}
Then we can rewrite
\begin{align*}
\P\left(M^+_t > - M^-_t \ \Big| \ \mathcal{F}_s\right)
= \P\left(G^+_t + \log Z_\infty > G^-_t + \log Z_\infty^- \ \bigg| \ \mathcal{G}_s\right).
\end{align*}
Theorem \ref{thm:main1d} tells us that $(G^+_t, G^-_t)$ conditioned on $\mathcal{G}_s$ converges in the double limit, first letting $t \to \infty$ and then $s \to \infty$, to a pair of independent standard Gumbel random variables. Thus the proof of Corollary \ref{cor:1d} is a consequence of the following lemma.

\begin{lemma}
Let $G_1, \ldots, G_n$ be independent standard Gumbel-distributed random variables. Then for any $a_1, \ldots, a_n$,
\begin{align*}
\P\Big(G_1 + a_1 \geq \max (a_2 + G_2, \ldots, a_n + G_n)\Big) = \frac{e^{a_1}}{\sum_{i=1}^n e^{a_i}}.
\end{align*}
\end{lemma}

\begin{proof}
Recall that the pdf of the standard Gumbel distribution is given by $e^{-(x+e^{-x})}$, and the cdf is given by $e^{-e^{-x}}$. Then by simple computations, setting $K := \log (1+\sum_{i=2}^n e^{a_i-a_1})$, we have 
\begin{align*}
& \P\Big(a_1+G_1 \geq \max (a_2 + G_2, \ldots, a_n + G_n)\Big) \\
= {} & \int_\R e^{-(g_1+e^{-g_1})} \prod_{i=2}^n e^{e^{-(a_1+g_1-a_i)}}\dd g_1 \\
= {} & \int_\R e^{-g_1-e^{-g_1}\left(1+\sum_{i=2}^n e^{a_i-a_1}\right)}\dd g_1\\
= {} & e^{-K} \int_\R e^{-(g_1-K)-e^{-(g_1-K)}}\dd g_1\\
= {} & \frac{1}{1+\sum_{i=2}^n e^{a_i-a_1}} = \frac{e^{a_1}}{\sum_{i=1}^n e^{a_i}}. \qedhere
\end{align*}
\end{proof}

We now prove the main theorem of this section.
\begin{proof}[Proof of Theorem \ref{thm:main1d}]
Let $y,z \geq 0$. Note that for all $0 \leq s \leq t$ we have
\begin{equation}
 \label{eqn:branchingProp1d}
 \P\left(M^+_t - m_t \leq y, -M^-_t - m_t \leq z \ \Big| \ \mathcal{G}_s\right) = \prod_{j \in \mathcal{N}_s} \nu_{s,t}(X_s(j),y,z)
\end{equation}
by the branching property, where we have set
\begin{align*}
  \nu_{s,t}(x,y,z) := \P\left(M^+_{t-s}-m_t \leq y-x, -M^-_{t-s}-m_t \leq z + x\right).
\end{align*}
We now bound $\nu_{s,t}$ from above and from below to obtain an asymptotically tight estimate for the joint cdf of $(M^+_t,M^-_t)$ given $\mathcal{G}_s$. We begin by computing a lower bound. Observe first that since $m_t - m_{t-s} = \sqrt{2}s + \frac{3}{2\sqrt{2}}\log \frac{t-s}{t}$, from the inequality $x/(x+1) \leq \log (1+x) \leq x $ we obtain that
\begin{align}
\sqrt{2}s - \frac{3}{2\sqrt{2}} \frac{s}{t-s} \leq m_t - m_{t-s} \leq \sqrt{2}s - \frac{3}{2\sqrt{2}} \frac{s}{t}.
\label{eq:diff_mt}
\end{align}
Therefore, noting that for any events $F,G$ one has $\P(F \cap G) \geq 1 - \P(F^c) - \P(G^c)$, we obtain that
\begin{align}
\begin{split}
\nu_{s,t}(x,y,z) \geq {}  1 &- \P\left(M^+_{t-s} - m_{t-s} \geq y - (x - \sqrt{2} s) - \frac{3}{2\sqrt{2}} \frac{s}{t-s}\right) \\
{} & - \P\left(-M^-_{t-s} - m_{t-s} \geq z + ( x + \sqrt{2} s) - \frac{3}{2\sqrt{2}} \frac{s}{t-s}\right).
\end{split}
\label{eq:psi_bound}
\end{align}
From \cite[Theorem 1]{Bramson1983} we know that $\P(M^+_t-m_t \geq x)$ converges uniformly as $t \to \infty$ to $\omega(x)$, which further satisfies
\[
  1 - \omega(x) \sim c_\star x e^{-\sqrt{2} x} \quad \text{ as } x \to \infty.
\]
Hence, \eqref{eq:psi_bound} yields
\begin{align*}
& \liminf_{t \to \infty} \prod_{j \in \mathcal{N}_s} \nu_{s,t}(X_s(j),y,z) \\
\geq {} & \prod_{j \in \mathcal{N}_s} \Big[1 - \omega(\sqrt{2} s - X_s(j) + y) - \omega(\sqrt{2}s + X_s(j) + z))\Big]\\
\geq {} & \prod_{j \in \mathcal{N}_s} \Big[1 - \Big\{c_\star(\sqrt{2} s - X_s(j) + y) e^{\sqrt{2}(X_s(j) - \sqrt{2} s - y)}\\
{} & + c_\star(\sqrt{2}s + X_s(j) + z))e^{\sqrt{2}(- X_s(j) - \sqrt{2}s - z))}\Big\}(1+\epsilon(s))\Big],
\end{align*}
where $s \mapsto \epsilon(s)$ is a random process such that $\lim_{s \to \infty} \epsilon(s) = 0$ a.s., where we used that $
\liminf_{s \to \infty} \min_{j \in \mathcal{N}_s}\left\{\sqrt{2} s  - |X_s(j)| \right\} = \infty$ a.s. This result follows plainly from the fact that the additive martingale converges to 0 a.s. which can be found in \cite{LS1987}.

Therefore, since for any numbers $a_i \in (0,1)^n$
\begin{align*}
  \prod_{i=1}^n (1-a_i) \geq e^{-\sum_i^n \frac{a_i}{1-a_i}} \geq e^{-\frac{1}{1-\max a_i}\sum_i^n a_i},
\end{align*}
and recalling also that 
\begin{align*}
  \lim_{s \to \infty} \max_{j \in \mathcal{N}_s} \left\{ |\sqrt{2} s - X_s(j)| e^{\sqrt{2}(X_s(j) - \sqrt{2} s)} \right\}= 0
\end{align*}
almost surely, we obtain from \eqref{eqn:branchingProp1d} that
\begin{align*}
& \liminf_{s \to \infty} \liminf_{t \to \infty} \P\left(M^+_t - m_t \leq y, -M^-_t - m_t \leq z \ \Big| \ \mathcal{F}_s\right) \\
\geq {} &  \liminf_{s \to \infty} \exp\Big(- c_\star\sum_{j \in \mathcal{N}_s}(y - (X_s(j) - \sqrt{2} s)) e^{-\sqrt{2}(y - (X_s(j) - \sqrt{2} s))} \\
& \quad \quad \quad \quad - c_\star \sum_{j \in \mathcal{N}_s} (z + ( X_s(j) + \sqrt{2} s))e^{-\sqrt{2}(z + ( X_s(j) + \sqrt{2} s))}\Big).
\end{align*}
Using again that the additive martingale $\sum_{j \in \mathcal{N}_s} e^{\sqrt{2} (X_s(j) - \sqrt{2} s)}$ converges to~$0$ a.s. we eventually obtain that
\begin{multline}
  \label{eqn:lbBranchingProp1d}
  \liminf_{s \to \infty} \liminf_{t \to \infty} \P\left(M^+_t - m_t \leq y, -M^-_t - m_t \leq z \ \Big| \ \mathcal{F}_s\right) \\
\geq \exp \left(-c_\star Z_\infty e^{-\sqrt{2}y} -c_\star Z_\infty^- e^{-\sqrt{2}z}\right).
\end{multline}

To obtain a similar upper bound, we use that for any pair of events $F,G$, $\P(F \cap G) = 1 - \P(F^c) - \P(G^c) + P(F^c \cap G^c)$, hence recalling \eqref{eq:diff_mt},
\begin{align*}
\begin{split}
\nu_{s,t}(x,y,z) \leq {} & 1 - \P\left(M^+_{t-s} - m_{t-s} \geq y - (x - \sqrt{2} s) - \frac{3}{2\sqrt{2}} \frac{s}{t}\right) \\
{} & - \P\left(-M^-_{t-s} - m_{t-s} \geq z + ( x + \sqrt{2} s) - \frac{3}{2\sqrt{2}} \frac{s}{t}\right)\\
& + \zeta_{s,t}(x,y,z),
\end{split}
\end{align*}
where
\begin{multline*}
\zeta_{s,t}(x,y,z) :=  \P \bigg(M^+_{t-s} - m_{t-s} \geq y - (x - \sqrt{2} s) - \frac{3}{2\sqrt{2}} \frac{s}{t},\\
-M^-_{t-s} - m_{t-s} \geq z + ( x + \sqrt{2} s) - \frac{3}{2\sqrt{2}} \frac{s}{t}\bigg).
\end{multline*}
Note that
\begin{align*}
\zeta_{s,t}(x,y,z) \leq {} & \P \bigg(M^+_{t-s} - m_{t-s} \geq y \wedge z  + \sqrt{2} s - \frac{3}{2\sqrt{2}} \frac{s}{t}\bigg),
\end{align*}
hence
\begin{align*}
\limsup_{t \to \infty} \zeta_{s,t}(x,y,z) \leq \omega(y \wedge z + \sqrt{2} s).
\end{align*}
As a result, with similar computations as in the proof of the lower bound, there exists a process $\epsilon(s)$ converging a.s. to $0$ as $s \to \infty$ such that,
\begin{align*}
\limsup_{t \to \infty} & \prod_{j \in \mathcal{N}_s} \nu_{s,t}(X_s(j),y,z) \\
\leq {} & \prod_{j \in \mathcal{N}_s} \Big[1 - \Big\{c_\star(\sqrt{2} s - X_s(j) + y) e^{\sqrt{2}(X_s(j) - \sqrt{2} s - y)}\\
{} & + c_\star(\sqrt{2}s + X_s(j) + z))e^{\sqrt{2}(- X_s(j) - \sqrt{2}s - z)}\\
{} & - c_\star(\sqrt{2}s + y \wedge z) e^{\sqrt{2}(- \sqrt{2}s - y \wedge z)}
\Big\}(1+\epsilon(s))\Big].
\end{align*}
Using that for any numbers $a_i < 1$, $\prod_{i=1}^n (1-a_i) \leq e^{-\sum_i^n a_i}$, and noting that for any $C \in \R$
\begin{align*}
\lim_{s \to \infty} \sum_{j \in \mathcal{N}_s} (\sqrt{2}s + C) e^{\sqrt{2}(-\sqrt{2} s - C)} = 0
\end{align*}
almost surely, we obtain that
\begin{align*}
& \limsup_{s \to \infty} \limsup_{t \to \infty} \P(M^+_t - m_t \leq y, -M^-_t - m_t \leq z \ | \ \mathcal{F}_s) \\
\leq {} & \exp (-c_\star Z_\infty e^{-\sqrt{2}y} -c_\star Z_\infty^- e^{-\sqrt{2}z}),
\end{align*}
which, together with \eqref{eqn:lbBranchingProp1d}, completes the proof.
\end{proof}

\begin{remark}
With similar computations to the ones made in the proof of Theorem \ref{thm:main1d} we would be able to prove joint convergence in distribution of
\[
  \left(\sum_{j \in \mathcal{N}_t} \delta_{X_t(j) - m_t}, \sum_{j \in \mathcal{N}_t} \delta_{-X_{t}(j) - m_t} \right)
\]
towards a pair of decorated Poisson point processes with random intensities $c_\star Z_\infty e^{-\sqrt{2} x} \dd x$ and $c_\star Z_\infty^- e^{-\sqrt{2}x}\dd x$ respectively, and such that these processes are independent conditionally on $(Z_\infty,Z_\infty^-)$. This result can be thought of as a unidimensional version of Conjecture \ref{con:extremal_process}.
\end{remark}

\section*{Acknowledgements}
We thank Alison Etheridge, Christina Goldschmidt and the referee for many helpful comments. A significant portion of the work was conducted while B.M. was invited at the University of Oxford, he gratefully acknowledges hospitality and the financial support. B.M. is partially supported by ANR grant MALIN (ANR-16-CE93-0003).
\bibliographystyle{plain}
\bibliography{Bibliography}

\begin{thebibliography}{10}

\bibitem{ABBS2013}
E.~A\"{\i}d\'{e}kon, J.~Berestycki, \'{E}. Brunet, and Z.~Shi.
\newblock Branching {B}rownian motion seen from its tip.
\newblock {\em Probability Theory and Related Fields}, 157(1-2):405--451, 2013.

\bibitem{ABK2013}
Louis-Pierre Arguin, Anton Bovier, and Nicola Kistler.
\newblock The extremal process of branching {B}rownian motion.
\newblock {\em Probability Theory and Related Fields}, 157(3-4):535--574, 2013.

\bibitem{BD1994}
Jean Bertoin and Ron~A. Doney.
\newblock On conditioning a random walk to stay nonnegative.
\newblock {\em Annals of Probability}, 22(4):2152--2167, 1994.

\bibitem{BM2019}
Pierre Boutaud and Pascal Maillard.
\newblock {A revisited proof of the Seneta-Heyde norming for branching random
  walks under optimal assumptions}.
\newblock {\em Electronic Journal of Probability}, 24, 2019.

\bibitem{Bramson1978}
Maury~D. Bramson.
\newblock {Maximal displacement of branching Brownian motion}.
\newblock {\em Communications on Pure and Applied Mathematics}, 31(5):531--581,
  1978.

\bibitem{Bramson1983}
Maury~D. Bramson.
\newblock {\em Convergence of Solutions of the Kolmogorov Equation to
  Travelling Waves}.
\newblock American Mathematical Society: Memoirs of the American Mathematical
  Society. American Mathematical Society, 1983.

\bibitem{CR1988}
Brigitte Chauvin and Alain Rouault.
\newblock {{KPP} equation and supercritical branching {B}rownian motion in the
  subcritical speed area. Application to spatial trees}.
\newblock {\em Probability Theory and Related Fields}, 80(2):299--314, 1988.

\bibitem{Chen2015}
Xinxin Chen.
\newblock A necessary and sufficient condition for the nontrivial limit of the
  derivative martingale in a branching random walk.
\newblock {\em Advances in Applied Probability}, 47(3):741--760, 2015.

\bibitem{Durrett2010}
Rick Durrett.
\newblock {\em Probability: Theory and Examples}.
\newblock Cambridge Series in Statistical and Probabilistic Mathematics.
  Cambridge University Press, 2010.

\bibitem{HH2006}
Robert Hardy and Simon~C. Harris.
\newblock A new formulation of the spine approach to branching diffusions.
\newblock {\em arXiv preprint math/0611054}, 2006.

\bibitem{HR2017}
Simon~C. Harris and Matthew~I. Roberts.
\newblock The many-to-few lemma and multiple spines.
\newblock {\em Annales de l'Institut Henri Poincaré, Probabilités et
  Statistiques}, 53(1):226--242, 2017.

\bibitem{Kyprianou2004}
Andreas~E. Kyprianou.
\newblock {Travelling wave solutions to the {KPP} equation: alternatives to
  Simon Harris' probabilistic analysis}.
\newblock {\em Annales de l'Institut Henri Poincaré, Probabilités et
  Statistiques}, 40(1):53--72, 2004.

\bibitem{LS1987}
Steven~P. Lalley and Thomas Sellke.
\newblock A conditional limit theorem for the frontier of a branching
  {B}rownian motion.
\newblock {\em Annals of Probability}, 15(3):1052--1061, 1987.

\bibitem{Lyons1997}
Russell Lyons.
\newblock A simple path to {B}iggins' martingale convergence for branching
  random walk.
\newblock In {\em Classical and Modern Branching Processes}, pages 217--221.
  Springer New York, 1997.

\bibitem{LPP1995}
Russell Lyons, Robin Pemantle, and Yuval Peres.
\newblock {Conceptual proofs of $L$ log $L$ criteria for mean behavior of
  branching processes}.
\newblock {\em Annals of Probability}, 23(3):1125--1138, 1995.

\bibitem{Mallein2015}
Bastien Mallein.
\newblock {Maximal displacement of d-dimensional branching Brownian motion}.
\newblock {\em Electronic Communications in Probability}, 20(0), 2015.

\bibitem{Novikov1997}
Aleksandr~A. Novikov.
\newblock Martingales, {T}auberian theorem, and strategies of gambling.
\newblock {\em Theory of Probability {\&} Its Applications}, 41(4):716--729,
  1997.

\bibitem{Roberts2010}
Matthew~I. Roberts.
\newblock {Spine changes of measure and branching diffusions (Ph.D. thesis)},
  2010.

\bibitem{SZ2015}
Eliran Subag and Ofer Zeitouni.
\newblock {Freezing and decorated Poisson point processes}.
\newblock {\em Communications in Mathematical Physics}, 337(1):55--92, 2015.

\bibitem{Uchiyama1980}
Kohei Uchiyama.
\newblock {Brownian first exit from and sojourn over one sided moving boundary
  and application}.
\newblock {\em Zeitschrift f{\"{u}}r Wahrscheinlichkeitstheorie und Verwandte
  Gebiete}, 54(1):75--116, 1980.

\bibitem{YR2011}
Ting Yang and Yan-Xia Ren.
\newblock {Limit theorem for derivative martingale at criticality w.r.t.
  branching Brownian motion}.
\newblock {\em Statistics \& Probability Letters}, 81(2):195--200, 2011.

\end{thebibliography}
\end{document}